\documentclass[a4paper,fleqn]{cas-sc}

\usepackage[numbers, sort&compress]{natbib}

\def\tsc#1{\csdef{#1}{\textsc{\lowercase{#1}}\xspace}}
\tsc{WGM}
\tsc{QE}
\tsc{EP}
\tsc{PMS}
\tsc{BEC}
\tsc{DE}
\begin{document}
\let\WriteBookmarks\relax
\def\floatpagepagefraction{1}
\def\textpagefraction{.001} 
\shorttitle{Influenza--Pneumococcal Co-infection Model}
\shortauthors{A.D. Fome et~al.}




\title[
  mode=title 
]{%
  When Reducing Transmission Is Not Enough: Carrier Dynamics 
  and Backward Bifurcation in Influenza--Pneumococcal Co-infection
}


%

\author[1]{Anna Daniel Fome}
%

\affiliation[1]{organization={Department of Economics. Mathematics and Statistics , Jordan University College},
	city={Morogoro},
	postcode={P.O Box 1878}, 
	country={Tanzania}}
\ead{anna.fome@juco.ac.tz}

\author[2]{Wolfgang Bock}
 \cormark[1]

\affiliation[inst2]{organization={Mathematics Department},
	addressline={Linnaeus Unversity, Universitetsplatsen 1}, 
	city={Växjö},
	postcode={352 52 },
	country={Sweden}}
\ead{wolfgang.bock@lnu.se}

\author[3]{Thomas Götz}
\ead{goetz@uni-koblenz.de }
\affiliation[3]{organization={Mathematical Institute, University of Koblenz},
                postcode={D-56070 }, 
                country={Germany}}
%

\newtheorem{theorem}{Theorem}
\newtheorem{proposition}[theorem]{Proposition}
\newtheorem{lemma}[theorem]{Lemma}
\newdefinition{remark}{Remark}
\newproof{proof}{Proof}

\begin{abstract}
Bacterial co-infection, particularly with \emph{Streptococcus pneumoniae}, is a major driver of morbidity during influenza epidemics, and this burden falls disproportionately on sub-Saharan Africa, where both pathogens circulate year-round at high levels and pneumococcal carriage remains common and persistent across the region. Yet existing mathematical models rarely track a host's bacterial carriage status jointly with their full influenza history through recovery. We develop a twelve-compartment ODE model that classifies each host by viral status (susceptible, infected, recovered) and bacterial status (susceptible, carrier, actively infected, recovered), allowing carrier-to-active progression to depend on prior influenza exposure. We establish positivity, boundedness, and a positively invariant region, then derive the basic reproduction number $\mathcal{R}_0$ and quasi-endemic equilibria of a single disease, together with invasion reproduction numbers $\mathbf{Inv}^{\mathbf{B}}$ and $\mathbf{Inv}^{\mathbf{V}}$ quantifying each pathogen's ability to invade a population where the other is already endemic. Center-manifold analysis shows that the bacterial subsystem and the full coexistence equilibrium undergo a backward bifurcation precisely when carrier transmission exceeds the carrier's total exit rate, $\beta_c>\kappa_S+\theta+\mu$, so that reducing $\mathcal{R}_0$ below one is not sufficient for elimination, while the viral subsystem bifurcates forward. A global sensitivity analysis (Latin Hypercube Sampling with partial rank correlation coefficients) shows that the viral and bacterial thresholds are governed by largely separate parameters, while the bacterial invasion potential is strongly shaped by viral transmission and recovery, giving a quantitative signature of cross-pathogen synergy.
\end{abstract}

\begin{highlights}
	\item A twelve-compartment model jointly tracks influenza and pneumococcal carriage status through recovery, motivated by the high dual burden of both pathogens in sub-Saharan Africa.
	\item Backward bifurcation occurs iff carrier transmission exceeds the carrier exit rate, $\beta_c>\kappa_S+\theta+\mu$, so $\mathcal{R}_0<1$ alone does not guarantee elimination.
	\item The viral subsystem bifurcates forward while the bacterial subsystem and full coexistence equilibrium bifurcate backward, isolating carriage as the source of bistability.
    \item Global PRCC sensitivity analysis shows viral and bacterial reproduction thresholds are driven by largely separate parameters.
	\item Bacterial invasion potential is strongly shaped by influenza transmission and recovery rates, giving a quantitative signature of cross-pathogen synergy.
	\item The backward-bifurcation condition offers a mechanistic explanation for the persistent vaccine-serotype pneumococcal carriage reported in Kenya, Malawi, and Tanzania after PCV introduction.
	\item Increasing spontaneous carriage clearance shifts the system from backward to forward bifurcation, offering a concrete intervention target for these settings.
\end{highlights}
 

\maketitle

\section{Introduction}
\label{sec:intro}

Sub-Saharan Africa bears a heavy burden of illness and death from two major respiratory germs: the influenza virus and Streptococcus pneumoniae (pneumococcus). Every year, seasonal flu causes between 290,000 and 650,000 deaths worldwide, with Sub-Saharan Africa suffering some of the highest mortality rates \cite{iuliano2018estimates} and reporting much more cases than regions like Latin America \cite{gessner2011seasonal}. Unlike temperate climates with seasonal winter spikes, Tanzanian surveillance data shows that flu circulates throughout the year without a single distinct peak \cite{mmbaga2012results,shedura2023evaluation}.

This constant presence allows the flu to persist, greatly increasing the risk of co-infection with pneumococcus—a combination that heavily affects young children and people living with HIV \cite{lagare2020epidemiology,madhi2000increased,madhi2002lower}. In 2015 alone, pneumococcal disease caused roughly 294,000 deaths among HIV-negative children under five years of age, mainly in low-income areas with limited healthcare and low vaccination rates \cite{wahl2018burden}. The public health burden of these pathogens is evident in the region: influenza and respiratory syncytial virus (RSV) contribute significantly to adult mortality in Soweto \cite{karstaedt2009influenza}, while pneumonia remains a leading cause of adult death in Kenya \cite{scott2000aetiology}. Understanding how flu and pneumococcus interact is not just an academic exercise—it is a critical public health priority that directly affects human lives.

This overlap stems from active interaction between the two pathogens, a phenomenon validated by population-level and physiological data. Drawing on two decades of hospital records, Shrestha et al.\ \cite{shrestha2013identifying} estimated that a previous influenza infection increases the probability of developing pneumococcal pneumonia by nearly one hundredfold, revealing that this relationship is predominantly unidirectional at the population level: influenza significantly encourages pneumococcal disease, but the reverse influence is considerably weaker, although potentially not completely negligible \cite{short2012mutually}. This imbalance has a mechanistic foundation within the host: influenza compromises the respiratory epithelium and reduces the clearance of pneumococcal bacteria \cite{richard2014tlr2,mccullers2006insights,nakamura2011synergistic,siegel2014influenza,aguilera2020inflammation}. Kinetic and immune-timing models indicate that the two pathogens can mutually reinforce each other, the timing of bacterial exposure in relation to influenza being crucial \cite{smith2013kinetics,smith2016critical,shrestha2013timedose} -- although no single mechanism completely accounts for the interaction \cite{mochan2023mathematical}. We discussed the population-level evidence for this interaction and its alignment with our proposed mechanism in Section~\ref{sec:discussion}.  The interaction is most pronounced during the carriage phase—the asymptomatic but contagious period prior to invasive pneumococcal disease \cite{weiser2018streptococcus}. Accordingly, a population-level model representing this dynamic should define carriage as a distinct stage and maintain a post-recovery record of each host's influenza history, necessitating an investigation into how this structure modifies the model's threshold behavior.

Answering that question requires the right analytical tools. Reproduction numbers and equilibrium stability can be examined through various methods—such as Lyapunov function constructions or numerical bifurcation continuation—but in compartmental systems like these, two techniques are used primarily: the next-generation matrix to determine epidemic thresholds \cite{van2002reproduction}, and the center-manifold method to investigate backward bifurcation \cite{castillo2004dynamical}—a situation where merely lowering the reproduction number below one does not suffice to eradicate the disease. We embrace these two here. This has been demonstrated primarily for HIV-tuberculosis co-infection, where the tuberculosis component can generate a backward bifurcation independently of the HIV component \cite{bhunu2009modeling,mekonen2022mathematical,omame2021fractional}. Our previous study of a single-host, two-illness scenario revealed a consistent trend, attributed to asymptomatic transmission \cite{fome2024nordic}. This paper asks whether this pattern persists and what conditions regulate it once carriage is framed epidemiologically.

Carriage is not the only pathway through which two pathogens can influence each other's threshold behavior; therefore, it is important to contextualize this method against others. A pertinent body of literature depicts two co-circulating \emph{viruses} through cross-immunity \cite{stowe2021interactions,nickbakhsh2019virus}, a methodology that we have also applied for influenza and SARS-CoV-2 \cite{fome2025analysis,fome2025interplay,fome2026dynamics}. That framework does not apply directly to this system. Unlike a second acute virus that shares immune overlap, a long-term bacterial pathogen in carriage exhibits entirely different dynamics. Rather than acting as an acute secondary infection competing for the same susceptible population, carriage functions as a persistent, largely asymptomatic reservoir. The model introduced here is explicitly tailored to this scenario, integrating an acute viral infection with a persistent bacterial infection in the carriage. Crucially, the model tracks each host's co-infection history throughout recovery. It allows the rate of bacterial development to vary depending on whether the host is influenza-naive, currently infected, or fully recovered (\(\kappa_S \neq \kappa_I \neq \kappa_R\); Section~\ref{sec:model}). This framework enables us to specify the invasion reproduction numbers $\mathbf{Inv}^{\mathbf{B}}$ and $\mathbf{Inv}^{\mathbf{V}}$ (Section~\ref{subsec:QDFE}), which quantify how one pathogen infiltrates a population where the other is already prevalent.

Existing influenza--pneumococcus models leave gaps that this structure is meant to address. Ali et al.\ \cite{ali2025model} derive a reproduction number of the same form used here and also report a backward bifurcation; based on their published model structure, their carrier compartment lets progression depend on current flu status, but has no separate compartment for carriers who have recovered from flu, so their model cannot represent that distinction, whatever their own derivation attributes the bifurcation to. P\'erez and Oluyori \cite{perez2022model} also track the history of joint infection through recovery, as we do, but classify the bacterial infection by severity rather than stage of carriage, so their model has no compartment for asymptomatic carriage at all. Beyond these structural gaps, the field faces a more general problem: a review of 72 mechanistic virus--bacteria co-infection models found that 79\% rely on interaction parameters that are assumed or borrowed rather than measured from data \cite{dhakal2025mechanistic}, and the present model shares this limitation for its own co-infection-specific parameters (Section~\ref{subsec:ParameterJustification}). We return to the structural comparison in the Discussion (Section~\ref{sec:discussion}) and to the parameter limitation in the Limitations subsection of the Conclusion (Section~\ref{sec:conclusion}).

Against this background, the central question of this paper is: does classifying bacterial status by carriage stage -- rather than severity, or without tracking joint history at all -- lead to qualitatively different epidemic thresholds, invasion conditions, and bifurcation behavior? Does the ability of carriers to transmit create its own route to backward bifurcation, so that reducing transmission alone may not be enough to eliminate disease?

The paper is organized as follows. Section~\ref {sec:model} presents the twelve-compartment model and establishes well-posedness. Section~\ref{sec:DFE} derives the basic reproduction number $\mathcal{R}_0$ and the stability of the disease-free equilibrium. Section~\ref{subsec:QDFE} derives the two quasi-endemic states of a single-disease and the numbers of invasions $\mathbf{Inv}^{\mathbf{B}}$ and $\mathbf{Inv}^{\mathbf{V}}$. Section~\ref{sec:CEE} extends the bifurcation analysis to the equilibrium of coexistence. Section~\ref{sec:Sensitivity} presents a global sensitivity analysis. Section~\ref{sec:discussion} discusses the key findings, implications for public-health, limitations, and future work. Throughout, we study the qualitative behavior of the deterministic model rather than fitting it to the data, so the numerical findings should be considered as illustrative. 

\section{Model Formulation}\label{sec:model}
 
Let $X_{ab}$, where $a \in \{S, I, R\}$ and $b \in \{S, C, I, R\}$, denote the compartment of individuals whose status with respect to \emph{viral infection} (influenza) is $a$, and whose status with respect to \emph{bacterial infection}  (pneumococcal) is $b$. The parameters $\beta$, $\gamma$ and $\iota$ denote, respectively, the transmission rate, the recovery rate, and the loss of immunity rate for the viral ($SIR$) of the model; all three are assumed to be independent of an individual's bacterial status. For bacterial disease, carriers and symptomatically infectious individuals transmit at rates $\beta_c$ and $\beta_I$, respectively. A \emph{carrier} is an individual who harbors the bacterium without yet displaying symptoms but can still infect susceptibles. A carrier progresses to the infectious class at a rate $\kappa_a$, which may depend on the individual's viral status $a$. Infectious bacteria-carriers recover at rate $\gamma_b$ and lose post-bacterial immunity at rate $\iota_b$.

Following standard practice in epidemic modeling \cite{martcheva2015introduction}, demographic processes (births at rate $\Lambda$ and natural deaths at rate $\mu$) are retained, but it is assumed that the disease evolves on a time scale on which demographic changes are negligible; see Lemma \ref{lem:Boundedness} for the exact statement. At any time $t$, the total population is

\begin{equation}
  N(t) \;=\; \sum_{a,\,b} X_{ab}(t),
  \label{eqn:TotalPop}
\end{equation}
which expands explicitly as
\begin{equation}
  N \;=\; X_{SS}+X_{SC}+X_{SI}+X_{SR}
        +X_{IS}+X_{IC}+X_{II}+X_{IR}
        +X_{RS}+X_{RC}+X_{RI}+X_{RR}.
  \label{eqn:Top_Pop}
\end{equation}
The twelve compartments are described in Table \ref{tab:Modelpopulation}, and the time argument $t$ is suppressed for brevity throughout, although its dependence remains implicit.

\begin{table}[!ht]
  \centering
  \caption{Description of the twelve model compartments.}
  \label{tab:Modelpopulation}
  \begin{tabular}{@{}ll@{}}
    \hline
    \textbf{Symbol} & \textbf{Description} \\
    \hline
    $X_{SS}$ & Susceptible to both viral and bacterial infections \\
    $X_{SC}$ & Susceptible to virus; carries bacteria (asymptomatic) \\
    $X_{SI}$ & Susceptible to virus; symptomatically infected with bacteria \\
    $X_{SR}$ & Susceptible to virus; recovered from bacterial infection \\
    $X_{IS}$ & Infected with virus; susceptible to bacteria \\
    $X_{IC}$ & Infected with virus; carries bacteria (asymptomatic) \\
    $X_{II}$ & Co-infected (virus and symptomatic bacterial infection) \\
    $X_{IR}$ & Infected with virus; recovered from bacterial infection \\
    $X_{RS}$ & Recovered from virus; susceptible to bacteria \\
    $X_{RC}$ & Recovered from virus; carries bacteria (asymptomatic) \\
    $X_{RI}$ & Recovered from virus; symptomatically infected with bacteria \\
    $X_{RR}$ & Recovered from both infections \\
    \hline
  \end{tabular}
\end{table}

The schematic flow diagram is shown in Figure \ref{fig:flowchart}.Note that, in general, $X_{ab} \neq X_{ba}$.

\begin{figure}[pos=ht] 
  \centering
  \includegraphics[width=\linewidth]{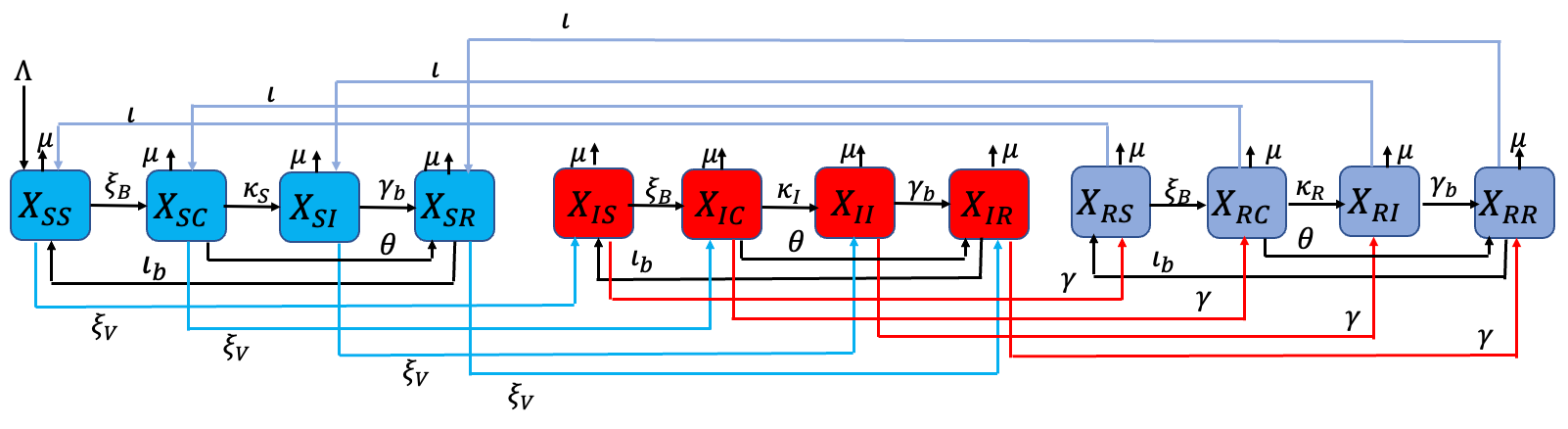}
  \caption{Schematic diagram of the co-infection model.  Arrows indicate
           transitions between compartments; rates are defined in the text.}
  \label{fig:flowchart}
\end{figure} 

\paragraph{Force of infection.} All virally infected individuals $(X_{IS},\,X_{IC},\,X_{II},\,X_{IR})$ transmit the virus per-capita $\beta/N$.  The resulting viral force of infection is

\begin{equation}
  \xi_V \;=\; \frac{\beta}{N}\sum_{b} X_{Ib}
          \;=\; \frac{\beta}{N}\bigl(X_{IS}+X_{IC}+X_{II}+X_{IR}\bigr).
  \label{eqn:viral_force_of_infection}
\end{equation}
Both bacterial carriers and symptomatically infectious individuals contribute to bacterial transmission, independently of their viral status, giving the bacterial force of infection
\begin{equation}
  \xi_B \;=\; \frac{\beta_c}{N}\sum_{a} X_{aC}
            + \frac{\beta_I}{N}\sum_{a} X_{aI}
          \;=\; \frac{\beta_c}{N}\bigl(X_{SC}+X_{IC}+X_{RC}\bigr)
              + \frac{\beta_I}{N}\bigl(X_{SI}+X_{II}+X_{RI}\bigr).
  \label{eqn:bacterial_force_of_infection}
\end{equation}

\noindent Combining the transition rates described above yields the following 12-dimensional ODE system:
\begin{equation}
  \begin{aligned}
    X_{SS}' &= \Lambda - (\xi_V+\xi_B)\,X_{SS}
               + \iota\, X_{RS} + \iota_b\, X_{SR} - \mu\, X_{SS}, \\
    X_{SC}' &= \xi_B\, X_{SS} - \xi_V\, X_{SC}
               + \iota\, X_{RC} - j_1\, X_{SC}, \\
    X_{SI}' &= \kappa_S\, X_{SC} - \xi_V\, X_{SI}
               + \iota\, X_{RI} - j_2\, X_{SI}, \\
    X_{SR}' &= \gamma_b\, X_{SI} + \theta\, X_{SC}
               - \xi_V\, X_{SR} + \iota\, X_{RR} - j_3\, X_{SR}, \\[4pt]
    X_{IS}' &= \xi_V\, X_{SS} - \xi_B\, X_{IS}
               + \iota_b\, X_{IR} - j_4\, X_{IS}, \\
    X_{IC}' &= \xi_V\, X_{SC} + \xi_B\, X_{IS} - j_5\, X_{IC}, \\
    X_{II}' &= \xi_V\, X_{SI} + \kappa_I\, X_{IC} - j_6\, X_{II}, \\
    X_{IR}' &= \xi_V\, X_{SR} + \gamma_b\, X_{II}
               + \theta\, X_{IC} - j_7\, X_{IR}, \\[4pt]
    X_{RS}' &= \gamma\, X_{IS} - \xi_B\, X_{RS}
               + \iota_b\, X_{RR} - j_8\, X_{RS}, \\
    X_{RC}' &= \xi_B\, X_{RS} + \gamma\, X_{IC} - j_9\, X_{RC}, \\
    X_{RI}' &= \kappa_R\, X_{RC} + \gamma\, X_{II} - j_{10}\, X_{RI}, \\
    X_{RR}' &= \gamma_b\, X_{RI} + \gamma\, X_{IR}
               + \theta\, X_{RC} - j_{11}\, X_{RR},
  \end{aligned}
  \label{eqn:MainModel}
\end{equation}
where the compound rates are defined by
\begin{equation}
  \begin{aligned}
    j_1  &= \kappa_S+\theta+\mu, &
    j_2  &= \gamma_b+\mu, &
    j_3  &= \iota_b+\mu, \\
    j_4  &= \gamma+\mu, &
    j_5  &= \kappa_I+\gamma+\theta+\mu, &
    j_6  &= \gamma+\gamma_b+\mu, \\
    j_7  &= \gamma+\iota_b+\mu, &
    j_8  &= \iota+\mu, &
    j_9  &= \kappa_R+\theta+\iota+\mu, \\
    j_{10} &= \gamma_b+\iota+\mu, &
    j_{11} &= \iota+\iota_b+\mu. &
  \end{aligned}
  \label{eqn:j_definitions}
\end{equation}
The System~\eqref{eqn:MainModel} is supplemented with non-negative initial conditions
\begin{equation}
  X_{ab}(0) \geq 0 \quad
  \text{for all } a \in \{S,I,R\},\; b \in \{S,C,I,R\}.
  \label{eqn:Initial_data}
\end{equation}


\subsection{Parameter values and biological justification}
\label{subsec:ParameterJustification} 

A recent analysis of 72 virus-bacteria co-infection models found that 79\% of interaction parameters were assumed or taken from other sources rather than being directly measured \cite{dhakal2025mechanistic}. To ensure clarity, we categorize the twelve parameters of System~\eqref{eqn:MainModel} into three levels, as indicated in the final column of Table~\ref{tab:parameter_justification}, based on the extent to which data directly supports each one. A Tier A parameter is directly derived from data; solely the carriage clearance rate $\theta$ qualifies. A Tier B parameter lacks a direct estimate; instead, its value is derived from transforming a measured quantity into the necessary rate or is otherwise restricted to a reasonable range by data -- this is true for $\beta$, $\gamma$, $\beta_c$, $\iota$, and $\mu$. A Tier C parameter lacks any numerical estimate in the literature: only the direction of its influence is backed by evidence, making its value a modeling assumption instead of a measurement -- this is relevant to $\gamma_b$, $\kappa_S$, $\beta_I$, $\kappa_I$, $\kappa_R$, and $\iota_b$. In all three categories,  the values are maintained within biologically realistic ranges from existing literature, adjusted to a standardized daily time unit, and are not tailored to any particular outbreak. Because six of the twelve parameters belong to Tier C, we consider them a source of uncertainty and analyze them directly in the sensitivity analysis (Section~\ref{sec:Sensitivity})

\renewcommand{\arraystretch}{1.35}
\footnotesize

\begin{longtable}{@{}
  >{\RaggedRight\arraybackslash}p{0.055\linewidth}
  >{\RaggedRight\arraybackslash}p{0.16\linewidth}
  >{\RaggedRight\arraybackslash}p{0.4\linewidth}
  c c
  >{\RaggedRight\arraybackslash}p{0.28\linewidth}
@{}}

\caption{Empirical evidence, illustrative baseline values, and evidence tier for the influenza--pneumococcal co-infection model. All rates are
per day.}
\label{tab:parameter_justification} \\

\toprule
\textbf{Sym.} & \textbf{Description} & \textbf{Empirical basis}
& \textbf{Base} & \textbf{Tier} & \textbf{Source} \\
\midrule
\endfirsthead

\multicolumn{6}{c}{{\bfseries \tablename~\thetable{} -- Continued}} \\[0.4em]
\toprule
\textbf{Sym.} & \textbf{Description} & \textbf{Empirical basis}
& \textbf{Base} & \textbf{Tier} & \textbf{Source} \\
\midrule
\endhead

\midrule
\multicolumn{6}{r}{{Continued on next page\ldots}} \\
\endfoot

\bottomrule
\endlastfoot

$\beta$ & Influenza transmission rate
& Tanzanian sentinel surveillance shows sustained, seasonally
variable influenza circulation; a Kenyan dynamic transmission model
provides estimates of transmissibility.
& $0.50$ & B
& \cite{mmbaga2012results,shedura2023evaluation,dawa2020seasonal}
\\[0.5em]

$\gamma$ & Influenza recovery rate
& Infectious period of approximately 2--7 days
($\gamma = 1/D_I \approx 0.14$--$0.50$/day).
& $0.33$ & B
& \cite{biggerstaff2014estimates,nishiura2010pros}
\\[0.5em]

$\beta_c$ & Pneumococcal carrier transmission rate
& African longitudinal cohorts estimate carriage acquisition and
transmission; dynamic models calibrate transmission parameters
against observed carriage and disease.
& $0.015$ & B
& \cite{abdullahi2012rates,tigoi2012rates,melegaro2010dynamic}
\\[0.5em]

$\gamma_b$ & Active pneumococcal disease recovery rate
& No directly transferable estimate; carriage studies quantify
acquisition and clearance, not recovery from active disease.
& $0.10$ & C
& \cite{abdullahi2012rates,chaguza2021carriage}
\\[0.5em]

$\theta$ & Pneumococcal carriage clearance rate
& Mean carriage duration $\approx 31$ days
($\theta = 1/D_C \approx 0.032$/day).
& $0.032$ & A
& \cite{abdullahi2012rates,chaguza2021carriage}
\\[0.5em]

$\kappa_S$ & Carriage $\rightarrow$ disease (no influenza)
& Carriage is an established disease precursor, but carriage
prevalence alone does not identify the daily progression hazard.
& $0.005$ & C
& \cite{bogaert2004streptococcus,weiser2018streptococcus,ojal2017sustained}
\\[0.5em]

$\beta_I$ & Pneumococcal transmission during influenza
& Interaction between influenza and pneumococcal colonization is
well supported; no directly transferable population-level
transmission rate exists.
& $0.035$ & C
& \cite{smith2013kinetics,richard2014tlr2,shrestha2013timedose}
\\[0.5em]

$\kappa_I$ & Carriage $\rightarrow$ disease (during influenza)
& Susceptibility to pneumococcal pneumonia rises sharply and
transiently following influenza infection.
& $0.05$ & C
& \cite{shrestha2013timedose,shrestha2015role,mccullers2006insights}
\\[0.5em]

$\kappa_R$ & Carriage $\rightarrow$ disease (post-influenza)
& The influenza-associated rise in pneumococcal disease risk is
transient rather than permanent.
& $0.02$ & C
& \cite{shrestha2013timedose,shrestha2015role}
\\[0.5em]

$\iota$ & Loss of influenza immunity
& For a single-season model, immunity loss can reasonably be
suppressed.
& $0.00$ & B
& \cite{mmbaga2012results,shedura2023evaluation}
\\[0.5em]

$\iota_b$ & Loss of pneumococcal protection
& Repeated acquisition and serotype-dependent carriage support
temporary, heterogeneous protection, but not a unique loss rate.
& $0.005$ & C
& \cite{abdullahi2012rates,chaguza2021carriage,emgaard2019carriage,emgaard2024co}
\\[0.5em]

$\mu$ & Natural birth/death rate
& Demographic turnover occurs on a much longer time scale than
acute infection.
& $0.0001$ & B
& \cite{martcheva2015introduction}
\\

\end{longtable}

\normalsize
\renewcommand{\arraystretch}{1}

\subsection{Well-posedness: positivity, boundedness, and invariant region}
\label{subsec:WellPosedness}

Before analyzing the equilibria of System~\eqref{eqn:MainModel}, we verify that the model is biologically well posed.  We establish three
properties:
\begin{enumerate}[\roman*]
  \item \textbf{Positivity;}  Solutions with non-negative initial data
        remain non-negative for all $t>0$.
  \item \textbf{Boundedness;}  The total population $N(t)$ is bounded
        above for all $t\geq 0$.
  \item \textbf{Invariant region;}  There exists a compact positively invariant region $\Omega\subset\mathbb{R}^{12}_{+}$ under the flow of \eqref{eqn:MainModel}.
\end{enumerate}
Together, these guaranty that compartment sizes are never negative and that the total population does not grow without bound, so System~\eqref{eqn:MainModel} is epidemiologically meaningful.

\begin{lemma}[Positivity of solutions]\label{lem:Positivity}
Let the initial data satisfy Equation~\eqref{eqn:Initial_data}.  Then every
solution $(X_{SS},X_{SC},\ldots,X_{RR})$ of
System~\eqref{eqn:MainModel} remains non-negative for all $t>0$:
\begin{equation}
  X_{ab}(0)\geq 0
  \;\Longrightarrow\;
  X_{ab}(t)\geq 0
  \quad\text{for all }t>0.
  \label{eqn:WP_positivity_statement}
\end{equation}
\end{lemma}

\begin{proof}
The proof is given in Appendix \ref{app:Positivity}.  In each case, it can be verified that, whenever $X_{ab}=0$ with all other non-negative components, the derivative $X_{ab}'\geq 0$. Hence, the non-negative orthant $\mathbb{R}^{12}_{+}$ is positively invariant.
\end{proof}

\begin{lemma}[Population dynamics and boundedness]\label{lem:Boundedness}
Let $N(t)=\sum_{a,b}X_{ab}(t)$.  Then
\begin{equation}
  N'(t) = \Lambda - \mu\,N(t),
  \label{eqn:WP_Nprime}
\end{equation}
with a unique solution
\begin{equation}
  N(t) = \frac{\Lambda}{\mu}
         +\!\left(N(0)-\frac{\Lambda}{\mu}\right)e^{-\mu t}.
  \label{eqn:WP_Nsolution}
\end{equation}
In particular:
\begin{enumerate}[\roman*]
  \item $N(t)$ is bounded: $0<N(t)\leq\max\!\left\{N(0),\,\Lambda/\mu\right\}$
        for all $t\geq 0$.
  \item $N(t)\to\Lambda/\mu$ as $t\to\infty$, independently of initial
        conditions.
  \item If $N(0)=\Lambda/\mu$, then $N(t)=\Lambda/\mu$ for all $t\geq 0$.
\end{enumerate}
\end{lemma}

\begin{proof}
Summing all twelve equations of the System~\eqref{eqn:MainModel}, we compute $N'(t)$ term by term.  The birth term $\Lambda$ appears only in equation $X_{SS}$.  The death term $-\mu X_{ab}$ appears in every equation, contributing $-\mu\sum_{a,b}X_{ab}=-\mu N$.  All remaining
terms - involving $\xi_V$, $\xi_B$, $\gamma$, $\gamma_b$, $\kappa_S$, $\kappa_I$, $\kappa_R$, $\theta$, $\iota$, and $\iota_b$ — represent
flows between compartments and therefore cancel in pairs when summed over all compartments.  We verify the three main groups explicitly.

\medskip
\noindent\textit{Viral force-of-infection terms.}
Each outflow $-X_{ab}\,\xi_V$ from a non-infected compartment is matched by an equal inflow $+X_{ab}\,\xi_V$ to the corresponding viral-infected compartment; these contributions sum up to zero.

\noindent\textit{Bacterial force-of-infection terms.}
Similarly, each outflow $-X_{ab}\,\xi_B$ from a non-carrier compartment is matched by an equal inflow to the corresponding carrier compartment;
their sum is zero.

\noindent\textit{Recovery, progression, and immunity-loss terms.}
Each of $\gamma$, $\gamma_b$, $\kappa_S$, $\kappa_I$, $\kappa_R$, $\theta$, $\iota$, $\iota_b$ moves individuals from one compartment to
another, contributing a net zero to $N'$.

\medskip\noindent
Therefore, $N'(t)=\Lambda-\mu N(t)$, which is the linear ODE~\eqref{eqn:WP_Nprime}.  Its unique solution is~\eqref{eqn:WP_Nsolution}.  Since $e^{-\mu t}\in(0,1]$ for $t\geq 0$ and $\mu>0$, the solution is bounded above by $\max\{N(0),\,\Lambda/\mu\}$ and converges to $\Lambda/\mu>0$ as $t\to\infty$.
\end{proof}

\begin{theorem}[Positively invariant feasible region]\label{thm:InvariantRegion}
Define
\begin{equation}
  \Omega \;=\;
  \left\{
    (X_{SS},X_{SC},\ldots,X_{RR})\in\mathbb{R}^{12}_{+}
    \;\bigg|\;
    \sum_{a,b}X_{ab} \leq \frac{\Lambda}{\mu}
  \right\}.
  \label{eqn:WP_Omega}
\end{equation}
The region $\Omega$ is:
\begin{enumerate}[\roman*]
  \item compact (closed and bounded) in $\mathbb{R}^{12}$;
  \item positively invariant under the flow of System ~\eqref{eqn:MainModel};
  \item attracting: every solution starting in
        $\mathbb{R}^{12}_{+}\setminus\Omega$ enters $\Omega$ in finite
        time and remains there.
\end{enumerate}
Consequently, it is sufficient to study the dynamics
of System~\eqref{eqn:MainModel} restricted to $\Omega$.
\end{theorem}

\begin{proof}
\textbf{(i) Compactness.}
The non-negative orthant $\mathbb{R}^{12}_{+}$ is closed, and the constraint $N\leq\Lambda/\mu$ defines a closed half-space.  Their intersection $\Omega$ is therefore closed.  Since every component satisfies $0\leq X_{ab}\leq\Lambda/\mu$, the region is bounded and therefore compact.

\medskip
\textbf{(ii) Positive invariance.}
Let $(X_{SS}(0),\ldots,X_{RR}(0))\in\Omega$, so that $N(0)\leq\Lambda/\mu$.  By Lemma \ref{lem:Positivity}, $X_{ab}(t)\geq 0$
for all $t>0$.  By Lemma \ref{lem:Boundedness},
\[
  N(t) = \frac{\Lambda}{\mu}
         +\underbrace{\left(N(0)-\frac{\Lambda}{\mu}\right)}_{\leq\,0}
         e^{-\mu t}
         \;\leq\; \frac{\Lambda}{\mu}
         \quad\text{for all }t\geq 0.
\]
Hence $(X_{SS}(t),\ldots,X_{RR}(t))\in\Omega$ for all $t\geq 0$.

\medskip
\textbf{(iii) Attractivity.}
If $N(0)>\Lambda/\mu$, then by~\eqref{eqn:WP_Nsolution},
\[
  N(t)-\frac{\Lambda}{\mu}
  = \left(N(0)-\frac{\Lambda}{\mu}\right)e^{-\mu t}\to 0^{+}
  \quad\text{as }t\to\infty.
\]
Since $e^{-\mu t}$ is strictly decreasing, there exists a finite time $T^{*}>0$ such that $N(t)\leq\Lambda/\mu$ for all $t\geq T^{*}$.
Combined with $X_{ab}(t)\geq 0$ (Lemma \ref{lem:Positivity}), the solution enters $\Omega$ by time $T^{*}$ and remains there.

\medskip\noindent
Properties (i)--(iii) together show that $\Omega$ is compact and positively invariant and that the global attractor of
System~\eqref{eqn:MainModel} lies within $\Omega$. The existence, uniqueness, and continuous dependence on initial data for the Lipschitz
right-hand side of~\eqref{eqn:MainModel} on $\Omega$ then follow from the standard ODE theory \cite{perko2013differential}.
\end{proof}

\begin{remark}[Population normalization~{\cite{martcheva2015introduction}}]
\label{rem:NormalizedPop}
When $\Lambda=\mu$, Lemma \ref{lem:Boundedness} shows that $N(t)\to 1$ as $t\to\infty$; hereafter we treat $N=1$ (normalized population).
Following \cite{martcheva2015introduction}, the use of standard incidence allows the endemic steady-state equations to be written in terms of
\emph{proportions}.  We therefore introduce the lowercase notation
\[
  x_{ab} \;=\; \frac{X_{ab}}{N}, \qquad
  a\in\{s,i,r\},\; b\in\{s,c,i,r\},
\]
so that $\sum_{a,b}x_{ab}=1$.  Theorem \ref{thm:InvariantRegion} guarantees that $N(t)>0$ for all $t\geq 0$ (since $N(t)\geq(\Lambda/\mu)(1-e^{-\mu t})>0$ whenever $\Lambda,\mu>0$), ensuring the normalization is well-defined at every finite time.
\end{remark}

\section{Disease-Free Equilibrium and Local Stability}
\label{sec:DFE}

\paragraph{Disease-free equilibrium.}
In the normalized variables introduced in Remark \ref{rem:NormalizedPop}, if at time $t=0$ the entire population is susceptible, then System~\eqref{eqn:MainModel} admits the \emph{disease-free equilibrium} (DFE)
\begin{equation}
  x_{ab}^{0}
  = \bigl[x_{ss}^{0},\,x_{sc}^{0},\,x_{si}^{0},\,x_{sr}^{0},\,
           x_{is}^{0},\,x_{ic}^{0},\,x_{ii}^{0},\,x_{ir}^{0},\,
           x_{rs}^{0},\,x_{rc}^{0},\,x_{ri}^{0},\,x_{rr}^{0}\bigr]
  = \bigl[1,0,0,0,0,0,0,0,0,0,0,0\bigr].
  \label{eqn:DFE}
\end{equation}

\paragraph{Basic reproduction number.}
The basic reproduction number $\mathcal{R}_0$ quantifies the average number of secondary infections produced by a single typical infected individual introduced into a fully susceptible population over the course of its infectious period \cite{diekmann1990definition}. We compute $\mathcal{R}_0$ using the next-generation matrix method \cite{van2002reproduction}. The infected compartments are $x_{sc},\,x_{si},\,x_{is},\,x_{ic},\,x_{ii},\,x_{ir},\,x_{rc},\,x_{ri}$. After decomposing the right-hand side of~\eqref{eqn:MainModel} into new-infection terms $\mathcal{F}$ and transition terms $\mathcal{V}$, and linearizing  about~\eqref{eqn:DFE}, we obtain the matrices $F$ and $V$
(see Appendix \ref{app:RepNum} for explicit expressions). Computing $FV^{-1}$ yields the next-generation matrix
\begin{equation}
  FV^{-1}=
  \begin{pmatrix}
    \mathcal{R}_{0b}^{T1} & \mathcal{R}_{0b}^{I1} & 0 &
    \mathcal{R}_{0b}^{I2}+\mathcal{R}_{0b}^{C2} &
    \mathcal{R}_{0b}^{I3} & 0 &
    \mathcal{R}_{0b}^{I4}+\mathcal{R}_{0b}^{C1} &
    \mathcal{R}_{0b}^{I5} \\
    0 & 0 & 0 & 0 & 0 & 0 & 0 & 0 \\
    0 & 0 & \mathcal{R}_{0a}^{V1} & \mathcal{R}_{0a}^{V2} &
    \mathcal{R}_{0a}^{V3} & \mathcal{R}_{0a}^{V4} & 0 & 0 \\
    0 & 0 & 0 & 0 & 0 & 0 & 0 & 0 \\
    0 & 0 & 0 & 0 & 0 & 0 & 0 & 0 \\
    0 & 0 & 0 & 0 & 0 & 0 & 0 & 0 \\
    0 & 0 & 0 & 0 & 0 & 0 & 0 & 0 \\
    0 & 0 & 0 & 0 & 0 & 0 & 0 & 0
  \end{pmatrix},
  \label{eqn:FV_Inverse}
\end{equation}
where the complete definitions of all entries are given in Appendix \ref{app:RepNum}. The two leading quantities are
\begin{equation}
  \mathcal{R}_{0a}^{V1} = \frac{\beta}{j_4},
  \qquad 
  \mathcal{R}_{0b}^{T1} = \frac{\beta_c j_2 + \beta_I \kappa_S}{j_1 j_2}.
  \label{eqn:R0_main}
\end{equation}
Row 1 of $FV^{-1}$ shows the expected secondary bacterial cases produced by a single carrier or infectious individual; Row~3 gives the corresponding quantity for a viral-infected individual. All other rows are zero, reflecting the fact that only these two types of infectious infections are seeded in the DFE. The basic reproduction number of System~\eqref{eqn:MainModel} is the spectral radius of~\eqref{eqn:FV_Inverse}:
\begin{equation}
  \mathcal{R}_0
  \;=\; \rho\!\left(FV^{-1}\right)
  \;=\; \max\!\left\{\mathcal{R}_{0a}^{V1},\;\mathcal{R}_{0b}^{T1}\right\}.
  \label{eqn:ReprNum}
\end{equation}
We now establish the local stability of the DFE.

\begin{lemma}[Local asymptotic stability of the DFE]
\label{lem:DFE_LAS}
If $\mathcal{R}_{0a}^{V1}<1$ and $\mathcal{R}_{0b}^{T1}<1$, then the disease-free equilibrium $x_{ab}^{0}$ defined by~\eqref{eqn:DFE} is locally asymptotically stable (LAS). In contrast, if $\mathcal{R}_0=\max\bigl\{\mathcal{R}_{0a}^{V1},\, \mathcal{R}_{0b}^{T1}\bigr\}>1$, then $x_{ab}^{0}$ is unstable.
\end{lemma}

\noindent The proof of Lemma \ref{lem:DFE_LAS} is deferred to Appendix \ref{app:Derivation of lem:DFE_LAS}. In the following remark, we discuss the interpretations of this lemma.


\begin{remark}[Biological interpretation of Lemma \ref{lem:DFE_LAS}]
\label{rem:DFE_interpretation}
The two conditions $\mathcal{R}_{0a}^{V1} < 1$ and
$\mathcal{R}_{0b}^{T1} < 1$ have clear epidemiological meanings.
\begin{itemize}
    \item $\mathcal{R}_{0a}^{V1} = \beta/j_4$ is the average number of secondary viral infections produced by one typical viral-infected individual in a fully susceptible population during the infectious period $1/j_4$.  The condition $\mathcal{R}_{0a}^{V1} < 1$ means each viral case generates fewer than one new case on average, so
    the viral outbreak cannot sustain itself.

    \item $\mathcal{R}_{0b}^{T1} = (\beta_c j_2 + \beta_I \kappa_S)/(j_1 j_2)$ combines two pathways of bacterial transmission.  The term $\beta_c j_2/(j_1 j_2) = \beta_c/j_1$ represents secondary infections generated by carriers before they either progress to active disease (rate $\kappa_S$), clear spontaneously (rate $\theta$), or die (rate $\mu$).  The term $\beta_I \kappa_S/(j_1 j_2)$ represents the infections generated during the active disease stage, weighted by the probability $\kappa_S/j_1$ that a carrier progresses to active infection.  The condition $\mathcal{R}_{0b}^{T1} < 1$ means the combined bacterial transmission is insufficient to sustain an epidemic.

    \item The threshold $\mathcal{R}_0 = \max\{\mathcal{R}_{0a}^{V1},
    \mathcal{R}_{0b}^{T1}\}$ is the overall basic reproduction number of the co-infection model.  Both disease must individually be below their respective thresholds for the disease-free state to be stable.  If either disease can sustain itself in the absence of the other, the DFE is unstable.
\end{itemize}
\end{remark}

\begin{figure}[pos=ht] 
  \centering
\includegraphics[width=.85\linewidth]{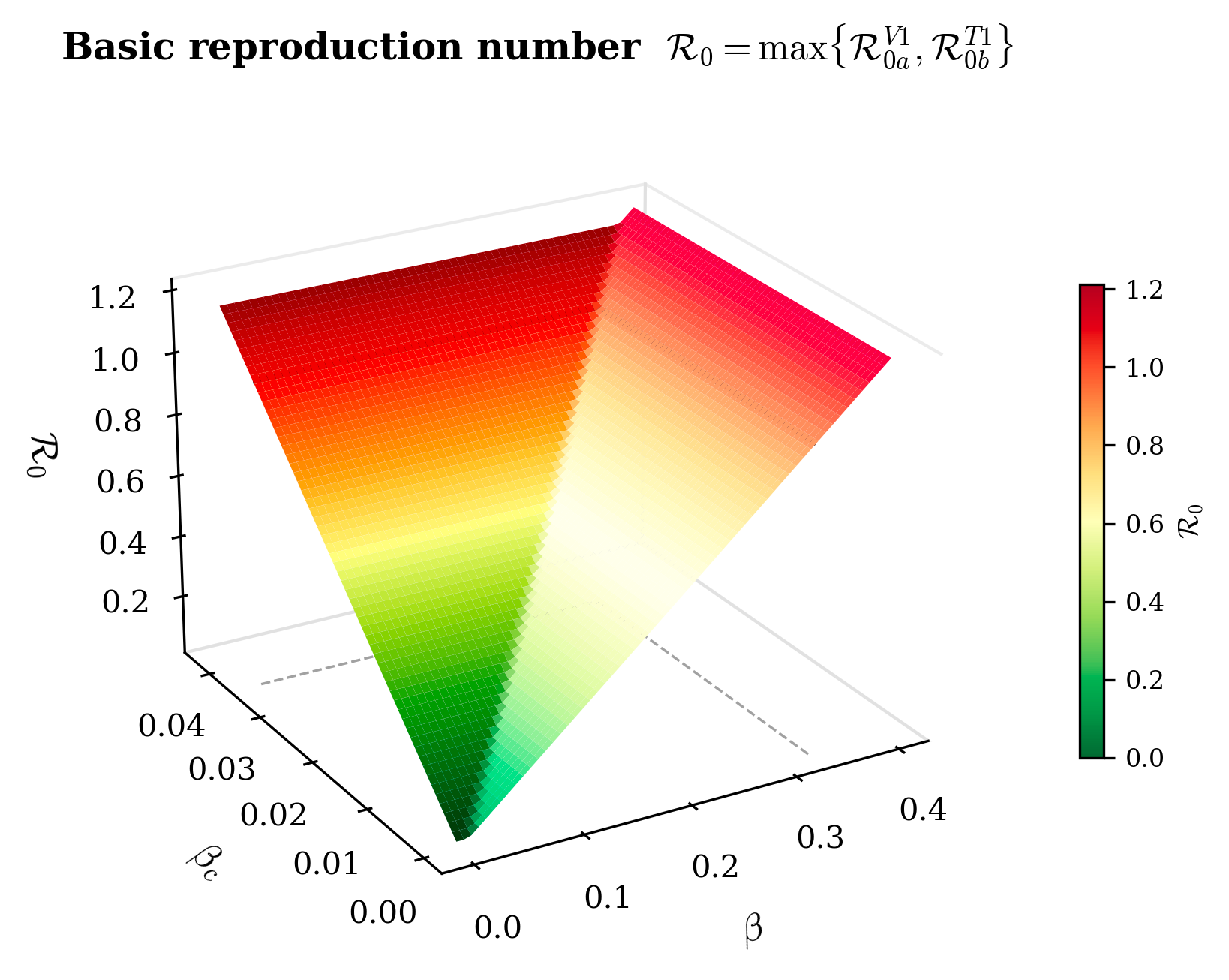} 
  \caption{Joint reproduction number $\mathcal{R}_0 = \max\{\mathcal{R}_{0a}^{V1}, \mathcal{R}_{0b}^{T1}\}$ across viral ($\beta \in [0.0, 0.4]$) and bacterial carrier ($\beta_c \in [0.0, 0.042]$) transmission rates. The dashed contour marks $\mathcal{R}_0=1$, the threshold for local stability of the disease-free equilibrium. Remaining parameters at baseline values (Table~\ref{tab:parameter_justification}), with $\beta_I=0.025$.}
  \label{fig:Max_Repr_Num}
\end{figure}  

Figure \ref{fig:Max_Repr_Num} shows $\mathcal{R}_0 = \max\{\mathcal{R}_{0a}^{V1},\mathcal{R}_{0b}^{T1}\}$ as a piecewise surface: each threshold depends only on its own pathogen's parameters ($\beta,\gamma$ for $\mathcal{R}_{0a}^{V1}$; $\beta_c,\theta,\kappa_S$ for $\mathcal{R}_{0b}^{T1}$), so each wing is flat with respect to the other axis, meeting at the diagonal where $\mathcal{R}_{0a}^{V1}=\mathcal{R}_{0b}^{T1}$ 
at these baseline parameters. The contour $\mathcal{R}_0=1$ (dashed) bounds the region in which both thresholds are below one, the necessary condition for local stability of the disease-free equilibrium (Lemma \ref{lem:DFE_LAS}).

\section{Quasi-Disease-Free Equilibria and Stability}
\label{subsec:QDFE}

We now study solutions of System~\eqref{eqn:MainModel} in which exactly one disease persists while the other is absent. Following \cite{long2008controlling,morris2010analysis}, we call such solutions \emph{quasi-disease-free equilibria} (QDFE): they are endemic with respect to one pathogen and free of the other. We denote by $x_{as}^{bf}$ and $x_{sb}^{vf}$ the quasi-bacteria-free equilibrium (QBFE) and the quasi-virus-free equilibrium (QVFE), respectively.

\subsection{Bacteria-free fixed point \texorpdfstring{$x_{as}^{bf}$}{} and its stability}
\label{subsec:BF_QDFE}

Setting all bacterial compartments to zero ($x_{sc}=x_{si}=x_{sr}=x_{ic}=x_{ii}=x_{ir}=x_{rc}=x_{ri}=x_{rr}=0$)
reduces System~\eqref{eqn:MainModel} to the three-dimensional viral-only system
\begin{equation}
  \begin{aligned}
    x_{ss}' &= \mu - \beta\, x_{is}\, x_{ss} + \iota\, x_{rs} - \mu\, x_{ss},\\
    x_{is}' &= \beta\, x_{is}\, x_{ss} - j_4\, x_{is},\\
    x_{rs}' &= \gamma\, x_{is} - j_8\, x_{rs},
  \end{aligned}
  \label{eqn:ReducedModel_NoBacterial}
\end{equation}
subject to initial conditions~\eqref{eqn:Initial_data} and normalization $x_{ss}+x_{is}+x_{rs}=1$.

\begin{lemma}[Existence and uniqueness of the QBFE]
\label{lem:QBFE_existence}
The System~\eqref{eqn:ReducedModel_NoBacterial} has a unique equilibrium of endimc bacteria-free $x_{as}^{bf}$, given by~\eqref{eqn:QBFE}--\eqref{eqn:Bacteria-free fixed}, if and only if $\mathcal{R}_{0a}^{V1}>1$. When $\mathcal{R}_{0a}^{V1}\leq 1$, the only non-negative equilibrium is the disease-free state $x_{ab}^{0}$.
\end{lemma}

\begin{proof}
See Appendix \ref{app:proof_QBFE_existence}.
\end{proof}

The unique QBFE is
\begin{equation}
  x_{as}^{bf}
  = \bigl[x_{ss}^{bf},\,0,\,0,\,0,\,x_{is}^{bf},\,0,\,0,\,0,\,
           x_{rs}^{bf},\,0,\,0,\,0\bigr],
  \label{eqn:QBFE}
\end{equation}
with non-zero components
\begin{equation}
  x_{ss}^{bf} = \frac{1}{\mathcal{R}_{0a}^{V1}},
  \qquad
  x_{is}^{bf} = \frac{j_8}{j_4+\iota}
                \!\left(1-\frac{1}{\mathcal{R}_{0a}^{V1}}\right),
  \qquad
  x_{rs}^{bf} = \frac{\gamma}{j_8}\,x_{is}^{bf}.
  \label{eqn:Bacteria-free fixed}
\end{equation}
These are biologically meaningful (positive) if and only if
$\mathcal{R}_{0a}^{V1}>1$.

\begin{lemma}[Local asymptotic stability of the QBFE] \label{lem:QBFE_stability}
If  $\mathcal{R}_{0a}^{V1}>1$, the equilibrium $x_{as}^{bf}$ is locally asymptotically stable with respect to
System~\eqref{eqn:ReducedModel_NoBacterial}.
\end{lemma}

\begin{proof}
See Appendix \ref{app:proof_QBFE_stability}.
The Jacobian of~\eqref{eqn:ReducedModel_NoBacterial} in $x_{as}^{bf}$ is the matrix $3\times3$
\begin{equation}
  \mathcal{J}^{bf}_{1} =
  \begin{pmatrix}
    -x_{is}^{bf}\beta - \mu & -x_{ss}^{bf}\beta & \iota \\
     x_{is}^{bf}\beta        &  x_{ss}^{bf}\beta - j_4 & 0 \\
    0                        & \gamma              & -j_8
  \end{pmatrix}.
  \label{eqn:J_bf_1}
\end{equation}
Its characteristic polynomial $P_1^{bf}(\lambda)=\lambda^3+a_1\lambda^2+a_2\lambda+a_3$ has all positive coefficients and satisfies $a_1a_2>a_3$ whenever $\mathcal{R}_{0a}^{V1}>1$; the Routh--Hurwitz criterion then gives
local asymptotic stability.
\end{proof}

\begin{theorem}[Forward bifurcation at $\mathcal{R}_{0a}^{V1}=1$] \label{thm:ReducedModel_NoBacterial_4ward} 
Under System~\eqref{eqn:ReducedModel_NoBacterial}, the bifurcation at $\mathcal{R}_{0a}^{V1}=1$ is supercritical (forward): the endemic equilibrium $x_{as}^{bf}$ bifurcates continuously from the DFE as $\mathcal{R}_{0a}^{V1}$ increases through~$1$.
\end{theorem}

\begin{proof}
See Appendix \ref{app:proof_bifurcation}. Applying Theorem~4.1 of \cite{castillo2004dynamical} with the bifurcation parameter $\beta^*=j_4$ (the value at which $\mathcal{R}_{0a}^{V1}=1$), the directional coefficients are
\[
  a = -\frac{j_8\,\beta^*}{j_4+\iota} < 0
  \qquad\text{and}\qquad
  b = 1 > 0,
\]
confirming a supercritical bifurcation.
\end{proof}

\begin{figure}[pos=ht] 
  \centering
\includegraphics[width=.85\linewidth]{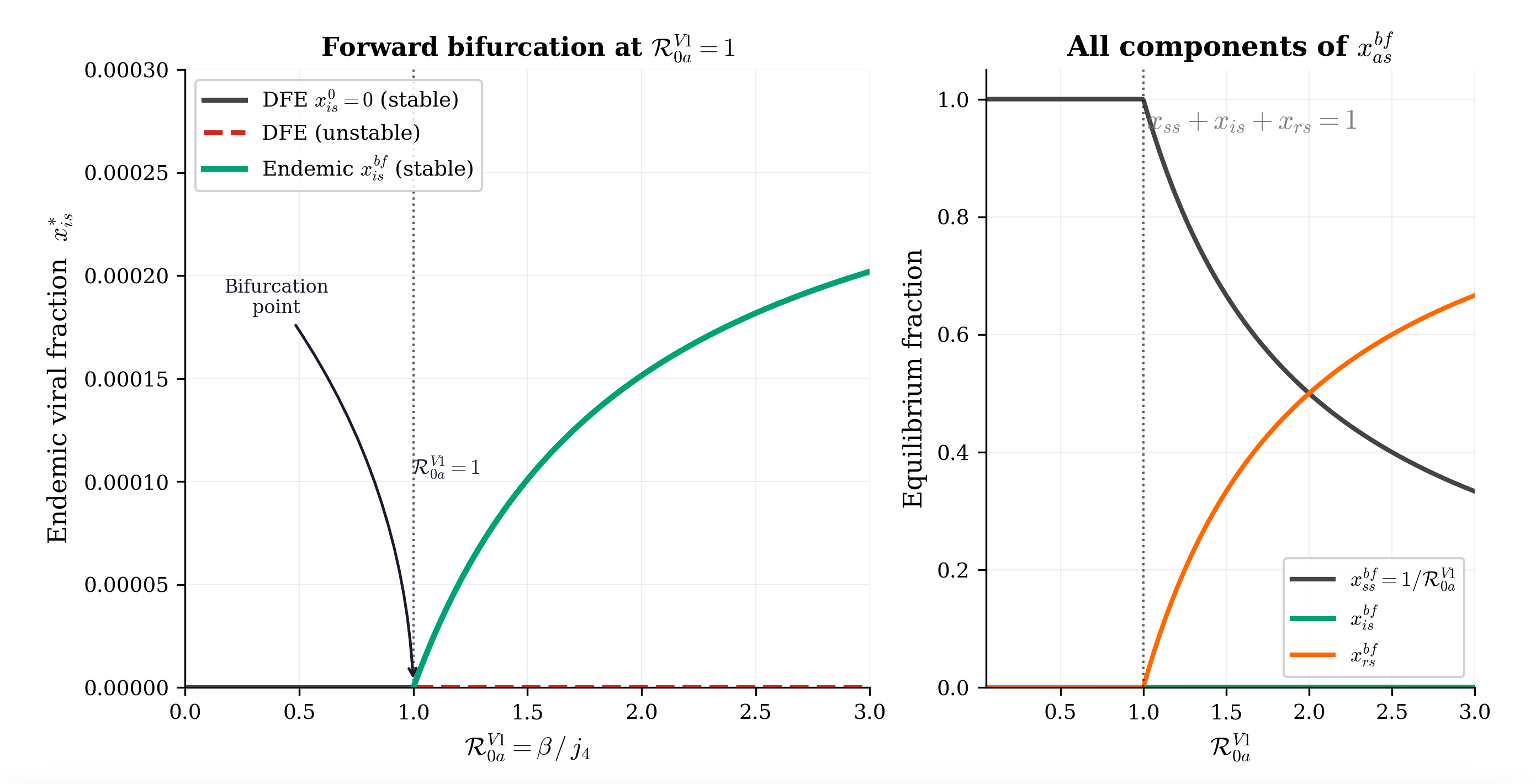} 
  \caption{Forward bifurcation diagram and equilibrium component profiles of the bacteria-free subsystem evaluated at fixed baseline values (\(\gamma = 0.33\), \(\mu = 0.0001\), \(\iota = 0.00\)).}
\label{fig:QBFE_Forward_bifurcation}
\end{figure}

Figure \ref{fig:QBFE_Forward_bifurcation} shows the forward transcritical bifurcation of Theorem \ref{thm:ReducedModel_NoBacterial_4ward}: the DFE (black) is stable for $\mathcal{R}_{0a}^{V1}<1$ and unstable (dashed) beyond it, exchanging stability with the endemic branch $x_{as}^{bf}$ (green) at $\mathcal{R}_{0a}^{V1}=1$. The infectious fraction $x_{is}^{bf}$ stays below $2\times10^{-4}$ throughout the plotted range, reflecting the high recovery rate $\gamma=0.33$; susceptibles decline as $x_{ss}^{bf}=1/\mathcal{R}_{0a}^{V1}$ (right panel) while recovered individuals rise correspondingly, consistent with $x_{ss}+x_{is}+x_{rs}=1$.

\subsection{Bacterial invasion of the virus-endemic state}
\label{subsec:BacteriaEmpire}

Having established that QBFE $x_{as}^{bf}$ exists and is locally asymptotically stable when $\mathcal{R}_{0a}^{V1}>1$, we now ask: \emph{Can bacterial disease invade a population in which viral disease is already endemic?} This is governed by the stability of $x_{as}^{bf}$ to bacterial perturbations in the full $12$-dimensional System~\eqref{eqn:MainModel}.

\paragraph{Block-triangular structure.}
The Jacobian of System~\eqref{eqn:MainModel} evaluated at $x_{as}^{bf}$ has the upper-triangular block form
\begin{equation}
  \mathcal{J}^{bf}\!\left(x_{as}^{bf}\right)
  =
  \begin{pmatrix}
    \mathcal{J}^{bf}_{1(3\times3)} & * & * \\
    \mathbf{0} & \mathcal{J}^{bf}_{2(6\times6)} & \mathbf{0} \\
    \mathbf{0} & * &   \mathcal{J}^{bf}_{3(3\times3)}
  \end{pmatrix},
  \label{eqn:BacteriaFreeMatrix}
\end{equation}
where $*$ denotes non-zero off-diagonal blocks. Because the matrix is an upper-triangular block, its spectrum is the union of the spectra of the three diagonal blocks \cite{horn2012matrix}. Lemma \ref{lem:QBFE_stability} establishes that $\mathcal{J}^{bf}_{1}$ is stable whenever $\mathcal{R}_{0a}^{V1}>1$, and Lemma \ref{lem:J3_stability} below shows that $\mathcal{J}^{bf}_{3}$ is stable under the same conditions. Consequently, the stability of $x_{as}^{bf}$ in the entire system reduces to the stability of the central block $\mathcal{J}^{bf}_{2}$.

\paragraph{The invasion reproduction number $\mathbf{Inv}^{\mathbf{B}}$.} The block $\mathcal{J}^{bf}_{2(6\times6)}$ governs the linearized dynamics of the six bacterial compartments $(x_{sc},\,x_{ic},\,x_{rc},\,x_{si},\,x_{ii},\,x_{ri})$ in a host population whose viral component is fixed in the endemic equilibrium $x_{as}^{bf}$. By the next-generation theorem \cite{van2002reproduction}, the stability
of $\mathcal{J}^{bf}_{2}$ is determined by the spectral radius of the corresponding next-generation matrix $F_1 V_1^{-1}$. The matrices $F_1$ and $V_1$ are given in Appendix \ref{ssst:Coinfected matrices}.

\begin{itemize}
  \item If $\rho(F_1 V_1^{-1})<1$, all the eigenvalues of $\mathcal{J}^{bf}_{2}$ have strictly negative real parts and the bacteria cannot invade;
  \item if $\rho(F_1 V_1^{-1})>1$, $\mathcal{J}^{bf}_{2}$ has at least one eigenvalue with positive real part and bacterial invasion occurs.
\end{itemize}

We define the \emph{bacterial invasion reproduction number}
\begin{equation}
  \mathbf{Inv}^{\mathbf{B}}
  \;=\; \rho\!\left(F_1 V_1^{-1}\right)
  \;=\; \mathcal{R}_{0b}^{T1}\,x_{ss}^{bf}
       + \mathcal{R}_{0b}^{T2}\,x_{is}^{bf}
       + \mathcal{R}_{0b}^{T3}\,x_{rs}^{bf},
  \label{eqn:BacteriaInvasion}
\end{equation}
where $\mathcal{R}_{0b}^{T1}$ is given in Equation \eqref{eqn:R0_main},
\begin{equation}
  \mathcal{R}_{0b}^{T2} = \mathcal{R}_{0b}^{I2}+\mathcal{R}_{0b}^{C2},
  \qquad
  \mathcal{R}_{0b}^{T3} = \mathcal{R}_{0b}^{I4}+\mathcal{R}_{0b}^{C1}.
  \label{eqn:R0bT2T3}
\end{equation}
The component reproduction numbers are given in Appendix \ref{app:RepNum} (equations~\eqref{eqn:R0bI2} to \eqref{eqn:R0bC1}).

$\mathbf{Inv}^{\mathbf{B}}$ is the average number of new bacterial infections started by a carrier or infectious individual entering a population where influenza is already endemic. It splits into three parts, one for each group of hosts at the QBFE: virally susceptible ($x_{ss}^{bf}$), virally infected ($x_{is}^{bf}$), and virally recovered ($x_{rs}^{bf}$). The first group behaves like an ordinary host, so the standard bacterial reproduction number $\mathcal{R}_{0b}^{T1}$ applies directly. The second group matters most: influenza weakens the immune response, so $\mathcal{R}_{0b}^{T2}>\mathcal{R}_{0b}^{T1}$, meaning bacterial spread is higher in currently-infected hosts than in ordinary susceptibles -- this is the co-infection pathway. The third group carries a smaller, residual effect from having had influenza and possibly losing that immunity again. Because $\mathcal{R}_{0b}^{T2}>\mathcal{R}_{0b}^{T1}$, an endemic influenza infection generally makes it easier for pneumococcus to invade: $\mathbf{Inv}^{\mathbf{B}}>\mathcal{R}_{0b}^{T1}$ is possible even when $\mathcal{R}_{0b}^{T1}<1$ on its own. This amplification effect grows with $x_{is}^{bf}$, which itself increases with the viral reproduction number $\mathcal{R}_{0a}^{V1}$. Table~\ref{tab:QBFE_Stability} summarises the resulting long-term outcomes.

\begin{table}[pos=ht]
  \centering
  \caption{Stability of the QBFE $x_{as}^{bf}$ and long-term outcomes.}
  \label{tab:QBFE_Stability}
  \begin{tabular}{@{}llll@{}}
    \hline
    \textbf{Region} & \textbf{Conditions} & \textbf{Long-term outcome} \\
    \hline
    R1 & $\mathcal{R}_{0a}^{V1}>1$,\; $\mathbf{Inv}^{\mathbf{B}}<1$
       & QBFE $x_{as}^{bf}$ stable; bacteria-free \\
    R2 & $\mathcal{R}_{0a}^{V1}>1$,\; $\mathbf{Inv}^{\mathbf{B}}>1$
       & QBFE $x_{as}^{bf}$ unstable; bacterial invasion \\
    R3 & $\mathcal{R}_{0a}^{V1}>1$,\; $\mathbf{Inv}^{\mathbf{B}}=1$
       & Bifurcation threshold; outcome determined by higher-order terms \\
    R4 & $\mathcal{R}_{0a}^{V1}\leq 1$
       & DFE $x_{ab}^{0}$ is the only non-negative equilibrium \\
    \hline
  \end{tabular}
\end{table}

\begin{lemma}[Stability of $\mathcal{J}^{bf}_{3}$]
\label{lem:J3_stability}
If $\mathcal{R}_{0a}^{V1}>1$, the sub-matrix
\begin{equation}
  \mathcal{J}^{bf}_{3} =
  \begin{pmatrix}
    -x_{is}^{bf}\beta - j_3 & 0      & \iota  \\
     x_{is}^{bf}\beta        & -j_7   & 0      \\
    0                        & \gamma & -j_{11}
  \end{pmatrix}
  \label{eqn:J_bf_3}
\end{equation}
is stable, i.e.\ all its eigenvalues have strictly negative real part.
\end{lemma}

\begin{proof}
See Appendix \ref{app:proof_J3_stability}. The characteristic polynomial is $P_3^{bf}(\lambda)=\lambda^3+c_1\lambda^2+c_2\lambda+c_3$, where
\begin{align*}
  c_1 &= x_{is}^{bf}\beta + j_3 + j_7 + j_{11}, \\
  c_2 &= x_{is}^{bf}\beta(j_7+j_{11}) + j_3 j_7 + j_3 j_{11} + j_7 j_{11}, \\
  c_3 &= x_{is}^{bf}\beta\,j_3(\gamma+j_{11}) + j_3 j_7 j_{11}.
\end{align*}
All three coefficients are strictly positive when $\mathcal{R}_{0a}^{V1}>1$, and the Routh--Hurwitz condition $c_1c_2>c_3$ holds by direct expansion.
\end{proof}

Although the bifurcation analysis in Figure \ref{fig:QBFE_Forward_bifurcation} outlines the long-term steady states of the bacteria-free subsystem as a function of reproductive fitness, it does not reveal the transient pathways or temporal behavior of the host population. To fully map out these underlying system dynamics, the corresponding state-space paths and explicit time-series simulations across the subcritical, critical, and epidemic regimes are compiled in the appendices. Specifically, Appendix \ref{app:QBFE_phase_space} (Figure \ref{fig:QBFE_Phase_Plane}) evaluates the geometric vector fields and phase-plane loops showing trajectory convergence, while Appendix \ref{app:QBFE_time_series} (Figure \ref{fig:QBFE_Time_series}) traces the explicit daily timelines of each population compartment under various initial conditions toward their mathematical equilibria.

Having fully characterized when bacteria can invade a virus-endemic population, we now address the complementary question: under what conditions can the viral disease invade a population in which the
bacterial disease is already endemic? This requires the analysis of the virus-free fixed point $x_{sb}^{vf}$, to which we turn next.

\subsection{Virus-free fixed point \texorpdfstring{$x_{sb}^{vf}$}{} and its stability}
\label{subsec:VF_QDFE}

Setting all viral-infected compartments to zero ($x_{is}=x_{ic}=x_{ii}=x_{ir}=0$, i.e.\ $x_{Ib}=0$) reduces System~\eqref{eqn:MainModel} to the four-dimensional bacterial-only system
\begin{equation}
  \begin{aligned}
    x_{ss}'(t) &= \mu
                  - (\beta_c x_{sc}+\beta_I x_{si})\,x_{ss}
                  + \iota_b\,x_{sr} - \mu\,x_{ss},\\
    x_{sc}'(t) &= (\beta_c x_{sc}+\beta_I x_{si})\,x_{ss}
                  - j_1\,x_{sc},\\
    x_{si}'(t) &= \kappa_S\,x_{sc} - j_2\,x_{si},\\
    x_{sr}'(t) &= \gamma_b\,x_{si} + \theta\,x_{sc} - j_3\,x_{sr},
  \end{aligned}
  \label{eqn:ReducedModel_NoVirus}
\end{equation}
subject to initial conditions~\eqref{eqn:Initial_data} and the normalization $x_{ss}+x_{sc}+x_{si}+x_{sr}=1$.

\begin{lemma}[Existence and uniqueness of the QVFE]
\label{lem:Unique_QVFE}
System~\eqref{eqn:ReducedModel_NoVirus} has a unique virus-free endemic equilibrium $x_{sb}^{vf}$, given by~\eqref{eqn:QVFE}--\eqref{eqn:Virus-free fixed}, if and only if $\mathcal{R}_{0b}^{T1}>1$. When $\mathcal{R}_{0b}^{T1}\leq 1$, the only non-negative equilibrium is the disease-free state $x_{ab}^{0}$.
\end{lemma}

\begin{proof}
See Appendix \ref{app:proof_QVFE_existence}.
\end{proof}

The unique QVFE is
\begin{equation}
  x_{sb}^{vf}
  = \bigl[x_{ss}^{vf},\,x_{sc}^{vf},\,x_{si}^{vf},\,x_{sr}^{vf},\,
           0,\,0,\,0,\,0,\,0,\,0,\,0,\,0\bigr],
  \label{eqn:QVFE}
\end{equation}
with non-zero components
\begin{equation}
  \begin{aligned}
    x_{ss}^{vf} &= \frac{1}{\mathcal{R}_{0b}^{T1}},\\[4pt]
    x_{sc}^{vf} &= \frac{(\mu+\iota_b)\bigl(j_1 j_2
                   +\iota_b(j_2+\kappa_S)\bigr)}{j_2}
                   \!\left(1-\frac{1}{\mathcal{R}_{0b}^{T1}}\right),\\[4pt]
    x_{si}^{vf} &= \frac{\kappa_S}{j_2}\,x_{sc}^{vf},\\[4pt]
    x_{sr}^{vf} &= \frac{\gamma_b\kappa_S+\theta j_2}{j_2 j_3}\,x_{sc}^{vf}.
  \end{aligned}
  \label{eqn:Virus-free fixed}
\end{equation}
These are biologically meaningful (positive) if and only if
$\mathcal{R}_{0b}^{T1}>1$.

\paragraph{Block-triangular structure.}
The Jacobian of System~\eqref{eqn:MainModel} evaluated at $x_{sb}^{vf}$ takes the block upper-triangular form
\begin{equation}
  \mathcal{J}^{vf}\!\left(x_{sb}^{vf}\right)
  =
  \begin{pmatrix}
    \mathcal{J}^{vf}_{1(4\times4)} & * & * \\
    \mathbf{0} & \mathcal{J}^{vf}_{2(4\times4)} & \mathbf{0} \\
    \mathbf{0} & * & \mathcal{J}^{vf}_{3(4\times4)}
  \end{pmatrix},
  \label{eqn:VirusFreeMatrix}
\end{equation}
where $*$ denotes non-zero off-diagonal blocks. By the same block-triangular spectral argument as in Section \ref{subsec:BF_QDFE} \cite{horn2012matrix}, the stability of $x_{sb}^{vf}$ reduces to the stability of the three diagonal blocks $\mathcal{J}^{vf}_{1}$, $\mathcal{J}^{vf}_{2}$, and $\mathcal{J}^{vf}_{3}$. Their explicit forms are
\begin{equation}
  \mathcal{J}^{vf}_{1} =
  \begin{pmatrix}
    -(x_{si}^{vf}\beta_I+x_{sc}^{vf}\beta_c+\mu) &
    -x_{ss}^{vf}\beta_c & -x_{ss}^{vf}\beta_I & \iota_b \\[2pt]
    x_{si}^{vf}\beta_I+x_{sc}^{vf}\beta_c &
    x_{ss}^{vf}\beta_c-j_1 & x_{ss}^{vf}\beta_I & 0 \\[2pt]
    0 & \kappa_S & -j_2 & 0 \\[2pt]
    0 & \theta   & \gamma_b & -j_3
  \end{pmatrix},
  \label{eqn:J_vf_1}
\end{equation}
\begin{equation}
  \mathcal{J}^{vf}_{2} =
  \begin{pmatrix}
    x_{ss}^{vf}\beta - x_{si}^{vf}\beta_I - x_{sc}^{vf}\beta_c - j_4 &
    x_{ss}^{vf}\beta & x_{ss}^{vf}\beta & x_{ss}^{vf}\beta+\iota_b \\[2pt]
    x_{sc}^{vf}\beta+x_{si}^{vf}\beta_I+x_{sc}^{vf}\beta_c &
    x_{sc}^{vf}\beta-j_5 & x_{sc}^{vf}\beta & x_{sc}^{vf}\beta \\[2pt]
    x_{si}^{vf}\beta & x_{si}^{vf}\beta+\kappa_I &
    x_{si}^{vf}\beta-j_6 & x_{si}^{vf}\beta \\[2pt]
    x_{sr}^{vf}\beta & x_{sr}^{vf}\beta+\theta &
    x_{sr}^{vf}\beta+\gamma_b & x_{sr}^{vf}\beta-j_7
  \end{pmatrix},
  \label{eqn:J_vf_2}
\end{equation}
\begin{equation}
  \mathcal{J}^{vf}_{3} =
  \begin{pmatrix}
    -(x_{si}^{vf}\beta_I+x_{sc}^{vf}\beta_c+j_8) &
    0 & 0 & \iota_b \\[2pt]
    x_{si}^{vf}\beta_I+x_{sc}^{vf}\beta_c & -j_9 & 0 & 0 \\[2pt]
    0 & \kappa_R & -j_{10} & 0 \\[2pt]
    0 & \theta   & \gamma_b & -j_{11}
  \end{pmatrix}.
  \label{eqn:J_vf_3}
\end{equation}

\subsubsection{Bifurcation analysis at $\mathcal{R}_{0b}^{T1}=1$}
\label{subsubsec:QVFE_Bifurcation}

The stability of $\mathcal{J}^{vf}_{1}$ is equivalent to the stability of the reduced bacterial System~\eqref{eqn:ReducedModel_NoVirus} at $x_{sb}^{vf}$, and is analyzed by the approach of \cite{castillo2004dynamical}. The computation, which follows the same steps as the proof of Theorem \ref{thm:ReducedModel_NoBacterial_4ward} (see Appendix \ref{app:proof_QVFE_bifurcation}), yields the directional coefficients
\begin{equation}
  a = \frac{
    (\gamma_b\iota_b + \gamma_b\kappa_S + \iota_b\kappa_S
     + \gamma_b\mu + \iota_b\mu + \kappa_S\mu + \mu^2
     + \gamma_b\theta + \mu\theta)(\gamma_b+\mu)(\kappa_S+\mu+\theta)
  }{
    (\beta_c - \kappa_S - \mu - \theta)(\iota_b+\mu)\,\kappa_S
  },
  \qquad b = 1 > 0.
  \label{eqn:QVFE_ab}
\end{equation}
The sign of $a$ depends critically on the denominator: $a>0$ if and only if $\beta_c > \kappa_S+\mu+\theta = j_1$, which is precisely the condition $\beta_c > j_1$.

\begin{theorem}[Bifurcation type at $\mathcal{R}_{0b}^{T1}=1$]
\label{thm:ReducedModel_NoVirus_Backward}
Consider System~\eqref{eqn:MainModel} restricted to the virus-free subspace ($x_{Ib}=0$), i.e.\ System~\eqref{eqn:ReducedModel_NoVirus}.
\begin{enumerate}[\roman*]
  \item If $\beta_c > \kappa_S+\mu+\theta$, then $a>0$ and the bifurcation at $\mathcal{R}_{0b}^{T1}=1$ is \textbf{backward} (subcritical): a stable endemic equilibrium coexists with the stable DFE for $\mathcal{R}_{0b}^{T1}$ slightly below~$1$, so the disease-free state is not globally stable even when $\mathcal{R}_{0b}^{T1}<1$.
  \item If $\beta_c \leq \kappa_S+\mu+\theta$, then $a\leq 0$ and the bifurcation is \textbf{forward} (supercritical): the unique endemic equilibrium $x_{sb}^{vf}$ is locally asymptotically stable whenever $\mathcal{R}_{0b}^{T1}>1$.
\end{enumerate}
\end{theorem}

\begin{proof}
Follows directly from Theorem~4.1 of \cite{castillo2004dynamical} applied to~\eqref{eqn:QVFE_ab}, with $b=1>0$ in both cases. The full eigenvector calculation is given in Appendix \ref{app:proof_QVFE_bifurcation}.
\end{proof}

\begin{remark}[Epidemiological significance of the backward bifurcation]
\label{rem:backward_bif}
The backward bifurcation condition $\beta_c>\kappa_S+\mu+\theta$ has a clear biological meaning: the carrier transmission rate $\beta_c$ exceeds the total per-capita rate at which carriers leave the infectious state (via progression $\kappa_S$, natural death $\mu$, or spontaneous clearance $\theta$). When this holds, a carrier is, on average, re-seeding more than one new carrier before leaving the infectious class, creating a self-sustaining carrier reservoir that persists even when $\mathcal{R}_{0b}^{T1}<1$. The public-health implication is that reducing $\mathcal{R}_{0b}^{T1}$ below~$1$ is necessary but \emph{not sufficient} to eliminate the bacterial disease; the carrier transmission rate $\beta_c$ must independently be reduced below $\kappa_S+\mu+\theta$.
\end{remark}


\begin{figure}[pos=ht] 
  \centering
\includegraphics[width=.85\linewidth]{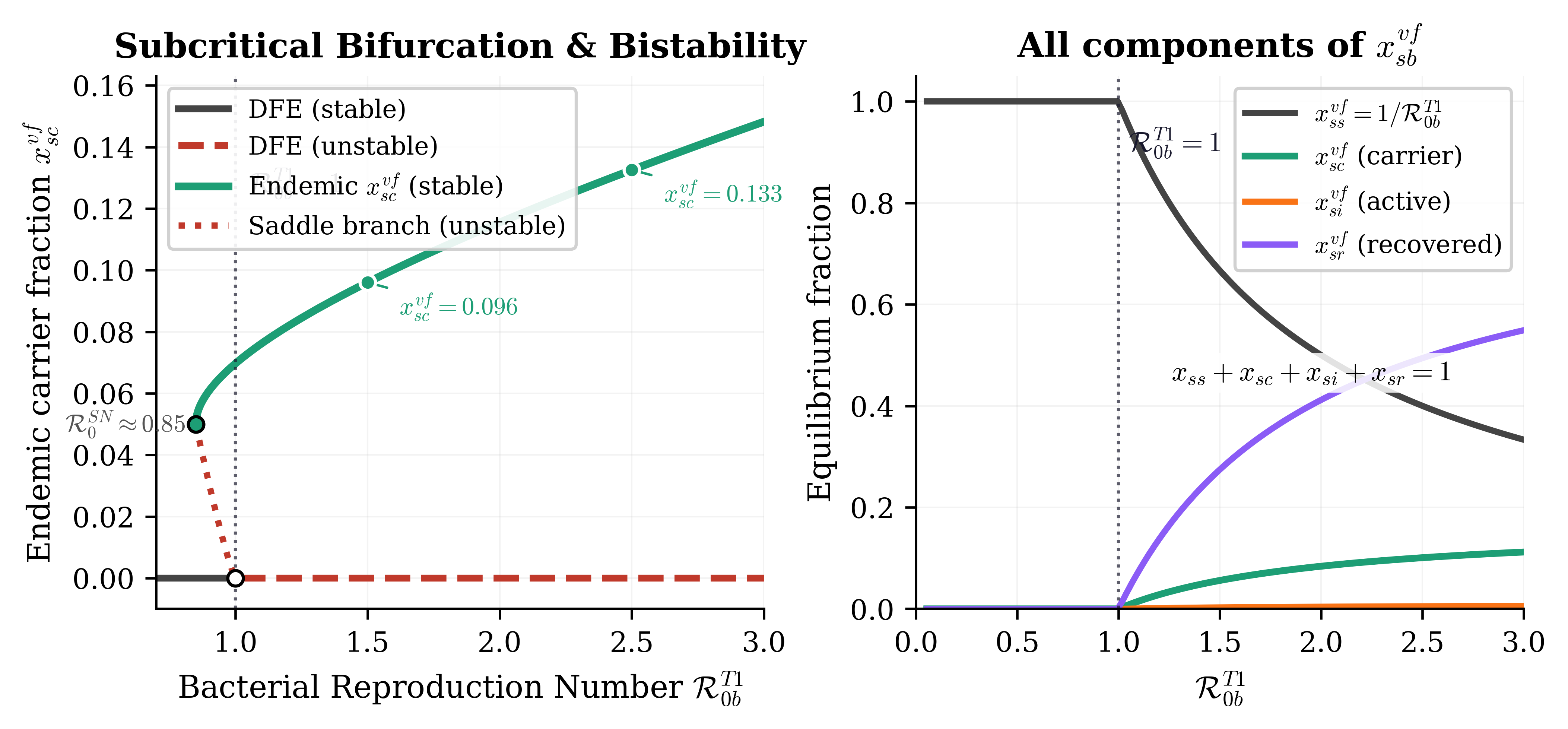} 
  \caption{Backward bifurcation of the bacterial-only subsystem when $\beta_c>\kappa_S+\mu+\theta$. (Left) Endemic carrier fraction $x_{sc}^{vf}$ vs.\ $\mathcal{R}_{0b}^{T1}$: stable (solid) and unstable (dashed) branches meet at the transcritical point ($\mathcal{R}_{0b}^{T1}=1$) and the endemic branch terminates at a saddle-node ($\mathcal{R}_0^{SN}\approx0.85$), giving a bistability window $0.85<\mathcal{R}_{0b}^{T1}<1$. (Right) Corresponding equilibrium fractions, satisfying $x_{ss}+x_{sc}+x_{si}+x_{sr}=1$. Parameters at baseline values (Table~\ref{tab:parameter_justification}), with $\theta=0.02$ and $\beta_I=0.025$.}
\label{fig:QVFE_bifurcation}
\end{figure}


Figure \ref{fig:QVFE_bifurcation} maps out the structural equilibrium state spaces of the isolated bacterial model subsystem under parameter regimes satisfying the subcritical inequality \(\beta_c > \kappa_S + \mu + \theta\), revealing a prominent backward bifurcation and an associated window of bistability within the localized transmission boundary of \(0.85 < \mathcal{R}_{0b}^{T1} < 1.0\). Inside this critical interval, a stable disease-free equilibrium coexists alongside a stable upper endemic equilibrium branch, separated by a continuous, backward-bending unstable saddle curve that collides cleanly with the axis floor at \(\mathcal{R}_{0b}^{T1} = 1.0\). Epidemiologically, this structure indicates that driving the basic reproduction number below unity is entirely insufficient to eradicate the pathogen, as a transient influx of carriers can push population densities above the saddle line into a self-sustaining endemic loop; absolute eradication instead requires suppressing transmission completely beneath the saddle-node threshold (\(\mathcal{R}_0^{SN} \approx 0.85\)) where the endemic branch collapses. Across this same transmission spectrum, the stratified component profiles show that the uninfected susceptible fraction decays steadily along a \(1/\mathcal{R}_{0b}^{T1}\) trajectory, acting as a recruitment channel that expands the stable asymptomatic carrier, active invasive, and recovered immune sub-populations while maintaining precise fractional conservation (\(\sum x_{sb}^{vf} = 1.0\)) across the entire continuum.

\subsubsection{Viral invasion of the bacterial-endemic state}
\label{subsubsec:VirusEmpire}

The central block $\mathcal{J}^{vf}_{2(4\times4)}$ governs the linearized dynamics of the four viral compartments $(x_{is},\,x_{ic},\,x_{ii},\,x_{ir})$ in a host population whose bacterial component is fixed at the endemic equilibrium $x_{sb}^{vf}$. As in Section \ref{subsec:BacteriaEmpire}, we apply the next-generation method \cite{van2002reproduction} to the sub-system consisting of all classes capable of transmitting the virus in the bacteria-endemic population: 

\begin{equation}
  \begin{aligned}
    x_{is}'(t) &= \xi_V\,x_{ss} - \xi_B\,x_{is}
                  + \iota_b\,x_{ir} - j_4\,x_{is},\\
    x_{ic}'(t) &= \xi_V\,x_{sc} + \xi_B\,x_{is} - j_5\,x_{ic},\\
    x_{ii}'(t) &= \xi_V\,x_{si} + \kappa_I\,x_{ic} - j_6\,x_{ii},\\
    x_{ir}'(t) &= \xi_V\,x_{sr} + \gamma_b\,x_{ii}
                  + \theta\,x_{ic} - j_7\,x_{ir}.
  \end{aligned}
  \label{eqn:ReducedModel_NoVirus2}
\end{equation}
Decomposing the right-hand side into new-infection terms
$\mathcal{F}_2$ and transition terms $\mathcal{V}_2$,
\begin{equation}
  \mathcal{F}_2 =
  \begin{pmatrix}
    \xi_V\,x_{ss}\\
    \xi_V\,x_{sc}\\
    \xi_V\,x_{si}\\
    \xi_V\,x_{sr}
  \end{pmatrix},
  \qquad
  \mathcal{V}_2 =
  \begin{pmatrix}
    \xi_B\,x_{is}-\iota_b\,x_{ir}+j_4\,x_{is}\\
     -\xi_B\,x_{is}+j_5\,x_{ic}\\
     -\kappa_I\,x_{ic}+j_6\,x_{ii}\\
     -\gamma_b\,x_{ii}-\theta\,x_{ic}+j_7\,x_{ir}
  \end{pmatrix},
  \label{eqn:ReducedModel_NoViruses_2_Matrices}
\end{equation}
and linearizing about $x_{sb}^{vf}$, we obtain the next-generation matrix $F_2 V_2^{-1}$ (see Appendix \ref{ssst:Coinfected matrices}
for $F_2$ and $V_2$).

\paragraph{The viral invasion reproduction number $\mathbf{Inv}^{\mathbf{V}}$.}
We define the \emph{viral invasion reproduction number} as the spectral
radius of $F_2 V_2^{-1}$:
\begin{equation}
  \mathbf{Inv}^{\mathbf{V}}
  \;=\; \rho(F_2 V_2^{-1})
  \;=\; x_{ss}^{vf}\,\mathcal{R}_{0a}^{V1}
       + x_{sc}^{vf}\,\mathcal{R}_{0a}^{V2}
       + x_{si}^{vf}\,\mathcal{R}_{0a}^{V3}
       + x_{sr}^{vf}\,\mathcal{R}_{0a}^{V4},
  \label{eqn:VirusInvasion}
\end{equation}
where the component viral reproduction numbers $\mathcal{R}_{0a}^{V1},\ldots,\mathcal{R}_{0a}^{V4}$ are given in Appendix \ref{app:RepNum} (equations~\eqref{eqn:R0aV1}--\eqref{eqn:R0aV4}). The stability of $\mathcal{J}^{vf}_{2}$ is then equivalent to $\mathbf{Inv}^{\mathbf{V}}<1$:
\begin{itemize}
  \item if $\mathbf{Inv}^{\mathbf{V}}<1$, all eigenvalues of
        $\mathcal{J}^{vf}_{2}$ have strictly negative real part and
        the virus cannot invade the bacterial-endemic state;
  \item if $\mathbf{Inv}^{\mathbf{V}}>1$, $\mathcal{J}^{vf}_{2}$ has
        at least one eigenvalue with positive real part and viral
        invasion occurs.
\end{itemize}

\paragraph{Biological interpretation of $\mathbf{Inv}^{\mathbf{V}}$.}
$\mathbf{Inv}^{\mathbf{V}}$ is the average number of new viral infections started by one infectious individual entering a population where pneumococcus is already endemic. As with $\mathbf{Inv}^{\mathbf{B}}$, it splits into weighted terms, one for each host group present at the QVFE: bacterially susceptible, carrier, actively infected, and recovered ($x_{ss}^{vf},x_{sc}^{vf},x_{si}^{vf},x_{sr}^{vf}$). The susceptible group behaves as an ordinary host, so $\mathcal{R}_{0a}^{V1}$ applies directly; the other three groups carry modified viral reproduction numbers ($\mathcal{R}_{0a}^{V2},\mathcal{R}_{0a}^{V3},\mathcal{R}_{0a}^{V4}$) reflecting how carriage, active bacterial disease, and bacterial recovery each alter a host's exposure to and susceptibility for influenza. This gives $\mathbf{Inv}^{\mathbf{V}}$ the same weighted-average structure as $\mathbf{Inv}^{\mathbf{B}}$, but with the roles of the two diseases reversed. The key asymmetry between them is that the bacterial disease can undergo a backward bifurcation (Theorem~\ref{thm:ReducedModel_NoVirus_Backward}) while the viral disease cannot (Theorem~\ref{thm:ReducedModel_NoBacterial_4ward}): the bacterial-endemic state $x_{sb}^{vf}$ can persist even when $\mathcal{R}_{0b}^{T1}<1$, so $\mathbf{Inv}^{\mathbf{V}}<1$ is necessary but not always sufficient to keep the virus out.

\begin{lemma}[Local asymptotic stability of the QVFE] \label{lem:QVFE_stability} If $\mathcal{R}_{0b}^{T1}>1$ and $\mathbf{Inv}^{\mathbf{V}}<1$, then the equilibrium $x_{sb}^{vf}$ is locally asymptotically stable with respect to System~\eqref{eqn:MainModel}. If $\mathbf{Inv}^{\mathbf{V}}>1$, then $x_{sb}^{vf}$ is unstable.
\end{lemma}

\begin{proof}
The block upper-triangular structure of~\eqref{eqn:VirusFreeMatrix}
reduces the stability question to the stability of $\mathcal{J}^{vf}_{1}$, $\mathcal{J}^{vf}_{2}$, and $\mathcal{J}^{vf}_{3}$ separately. Stability of $\mathcal{J}^{vf}_{1}$ under $\mathcal{R}_{0b}^{T1}>1$ follows from Lemma \ref{lem:Unique_QVFE} and  Theorem \ref{thm:ReducedModel_NoVirus_Backward}(ii). Stability of $\mathcal{J}^{vf}_{2}$ is equivalent to $\mathbf{Inv}^{\mathbf{V}}<1$ by the next-generation argument above. Stability of $\mathcal{J}^{vf}_{3}$ is proved in Appendix \ref{app:proof_QVFE_J3_stability}.
\end{proof}

The long-term outcomes for the QVFE are summarized in Table \ref{tab:QVFE_Stability}.

\begin{table}[!ht]
  \centering
  \caption{Stability of the QVFE $x_{sb}^{vf}$ and long-term outcomes.}
  \label{tab:QVFE_Stability}
  \begin{tabular}{@{}lll@{}}
    \hline
    \textbf{Region} & \textbf{Conditions} & \textbf{Long-term outcome} \\
    \hline
    R5 & $\mathcal{R}_{0b}^{T1}>1$,\; $\mathbf{Inv}^{\mathbf{V}}<1$
       & QVFE $x_{sb}^{vf}$ stable; virus-free \\
    R6 & $\mathcal{R}_{0b}^{T1}>1$,\; $\mathbf{Inv}^{\mathbf{V}}>1$
       & QVFE $x_{sb}^{vf}$ unstable; viral invasion \\
    R7 & $\mathcal{R}_{0b}^{T1}>1$,\; $\mathbf{Inv}^{\mathbf{V}}=1$
       & Bifurcation threshold; higher-order analysis required \\
    R8 & $\mathcal{R}_{0b}^{T1}\leq 1$,\; $\beta_c>\kappa_S+\mu+\theta$
       & Backward bifurcation: endemic state may persist
         despite $\mathcal{R}_{0b}^{T1}<1$ \\
    \hline
  \end{tabular}
\end{table}

\section{Coexistence Endemic Equilibrium and Local Stability}
\label{sec:CEE}

We now study the state in which \emph{both} diseases persist simultaneously — the coexistence endemic equilibrium (CEE). The analysis proceeds in two steps. First we show the CEE exists and can be written in closed form. Then we determine whether, near the epidemic threshold $\mathcal{R}_0=1$, the system undergoes a forward (supercritical) or backward (subcritical) bifurcation, since the direction determines whether reducing $\mathcal{R}_0$ below~$1$ is sufficient to eliminate co-infection.

\subsection{Existence of the CEE}
\label{subsec:CEE_existence}

Let
\begin{equation}
  x_{ab}^{*}
  = \bigl[x_{ss}^{*},\,x_{sc}^{*},\,x_{si}^{*},\,x_{sr}^{*},\,
           x_{is}^{*},\,x_{ic}^{*},\,x_{ii}^{*},\,x_{ir}^{*},\,
           x_{rs}^{*},\,x_{rc}^{*},\,x_{ri}^{*},\,x_{rr}^{*}\bigr]
  \label{eqn:EE}
\end{equation}
denote the CEE of System~\eqref{eqn:MainModel}. The full nonlinearity of the force-of-infection terms prevents a simple explicit formula in general. However, under permanent immunity ($\iota=\iota_b=0$), the equilibrium equations become linear in the endemic forces of infection $\xi_V^{*}$ and $\xi_B^{*}$, yielding the closed-form expressions in Appendix \ref{app:EE}. The local stability analysis in Section \ref{subsec:CEE_bifurcation} applies to the general case $\iota,\iota_b\geq 0$.

\begin{remark}[Continuity in $\iota$, $\iota_b$]
The bifurcation coefficients $a$ and $b$ (Section \ref{subsec:CEE_bifurcation}) are continuous functions of all parameters including $\iota$ and $\iota_b$. Hence the qualitative bifurcation type established at $\iota=\iota_b=0$ persists for sufficiently small immunity-loss rates.
\end{remark}

\subsection{Center manifold setup}
\label{subsec:CEE_setup}

We apply the center manifold theorem of Castillo-Chavez and Song \cite{castillo2004dynamical} (Theorem~4.1) to determine the local dynamics near $\mathcal{R}_0=1$.

\paragraph{Variable relabelling.}
To match the notation required by the theorem, we rename the state variables as shown in Table \ref{tab:variable_relabelling}.

\begin{table}[!ht]
  \centering
  \caption{Variable relabelling for the bifurcation analysis.}
  \label{tab:variable_relabelling}
  \begin{tabular}{@{}llll@{}}
    \hline
    New & Original & New & Original \\
    \hline
    $x_1$  & $x_{ss}$ & $x_7$    & $x_{ii}$ \\
    $x_2$  & $x_{sc}$ & $x_8$    & $x_{ir}$ \\
    $x_3$  & $x_{si}$ & $x_9$    & $x_{rs}$ \\
    $x_4$  & $x_{sr}$ & $x_{10}$ & $x_{rc}$ \\
    $x_5$  & $x_{is}$ & $x_{11}$ & $x_{ri}$ \\
    $x_6$  & $x_{ic}$ & $x_{12}$ & $x_{rr}$ \\
    \hline
  \end{tabular}
\end{table}

Writing $y_i = dx_i/dt$ and $Y=(y_i)^{\top}$, System~\eqref{eqn:MainModel} becomes
\begin{equation}
  \begin{aligned}
    y_1  &= \Lambda - x_1\xi_V - x_1\xi_B
             + \iota\,x_9 + \iota_b\,x_4 - \mu\,x_1,\\
    y_2  &= -x_2\xi_V + x_1\xi_B + \iota\,x_{10} - j_1\,x_2,\\
    y_3  &= -x_3\xi_V + \kappa_S\,x_2 + \iota\,x_{11} - j_2\,x_3,\\
    y_4  &= -x_4\xi_V + \gamma_b\,x_3 + \iota\,x_{12}
             + \theta\,x_2 - j_3\,x_4,\\
    y_5  &= x_1\xi_V - x_5\xi_B + \iota_b\,x_8 - j_4\,x_5,\\
    y_6  &= x_2\xi_V + x_5\xi_B - j_5\,x_6,\\
    y_7  &= x_3\xi_V + \kappa_I\,x_6 - j_6\,x_7,\\
    y_8  &= x_4\xi_V + \gamma_b\,x_7 + \theta\,x_6 - j_7\,x_8,\\
    y_9  &= -x_9\xi_B + \gamma\,x_5 + \iota_b\,x_{12} - j_8\,x_9,\\
    y_{10} &= x_9\xi_B + \gamma\,x_6 - j_9\,x_{10},\\
    y_{11} &= \gamma\,x_7 + \kappa_R\,x_{10} - j_{10}\,x_{11},\\
    y_{12} &= \gamma\,x_8 + \gamma_b\,x_{11} + \theta\,x_{10}
               - j_{11}\,x_{12},
  \end{aligned}
  \label{eqn:MainModel_NewNotation}
\end{equation}
with forces of infection
\begin{equation}
  \xi_V = \frac{\beta}{N}(x_5+x_6+x_7+x_8),
  \qquad
  \xi_B = \frac{\beta_c}{N}(x_2+x_6+x_{10})
          + \frac{\beta_I}{N}(x_3+x_7+x_{11}),
  \label{eqn:force_of_infections_Transformed}
\end{equation}
where $N=\sum_{i=1}^{12}x_i$.

\paragraph{Bifurcation parameter.}
We assume $\mathcal{R}_{0b}^{T1}<\mathcal{R}_{0a}^{V1}$ so that $\mathcal{R}_0 = \mathcal{R}_{0a}^{V1} = \beta/j_4$. Setting $\mathcal{R}_0=1$ gives the critical viral transmission rate
\begin{equation}
  \beta^{*} = j_4 = \gamma+\mu.
  \label{eqn:beta_star}
\end{equation}
Evaluating the Jacobian of~\eqref{eqn:MainModel_NewNotation} at the DFE $x_{ab}^{0}$ with $\beta=\beta^{*}$ yields the $12\times12$matrix
\begin{equation}
  \setcounter{MaxMatrixCols}{12}
  \mathcal{J}_{\mathcal{T}} =
  \begin{pmatrix}
    -\mu  & -\beta_c & -\beta_I & \iota_b   & -\beta^* & -j^*_{12} & -j^*_{13} & -\beta^* & \iota & -\beta_c & -\beta_I & 0        \\
    0     & \beta_c-j_1 & \beta_I & 0        & 0        & \beta_c   & \beta_I   & 0        & 0     & \beta_c+\iota & \beta_I & 0    \\
    0     & \kappa_S & -j_2     & 0         & 0        & 0         & 0         & 0        & 0     & 0        & \iota    & 0        \\
    0     & \theta   & \gamma_b & -j_3      & 0        & 0         & 0         & 0        & 0     & 0        & 0        & \iota    \\
    0     & 0        & 0        & 0         & 0        & \beta^*   & \beta^*   & \beta^*+\iota_b & 0 & 0   & 0        & 0        \\
    0     & 0        & 0        & 0         & 0        & -j_5      & 0         & 0        & 0     & 0        & 0        & 0        \\
    0     & 0        & 0        & 0         & 0        & \kappa_I  & -j_6      & 0        & 0     & 0        & 0        & 0        \\
    0     & 0        & 0        & 0         & 0        & \theta    & \gamma_b  & -j_7     & 0     & 0        & 0        & 0        \\
    0     & 0        & 0        & 0         & \gamma   & 0         & 0         & 0        & -j_8  & 0        & 0        & \iota_b  \\
    0     & 0        & 0        & 0         & 0        & \gamma    & 0         & 0        & 0     & -j_9     & 0        & 0        \\
    0     & 0        & 0        & 0         & 0        & 0         & \gamma    & 0        & 0     & \kappa_R & -j_{10}  & 0        \\
    0     & 0        & 0        & 0         & 0        & 0         & 0         & \gamma   & 0     & \theta   & \gamma_b & -j_{11}
  \end{pmatrix},
  \label{eqn:Bifurcation_Matrix}
\end{equation}
where $j^*_{12}=\beta^*+\beta_c$ and $j^*_{13}=\beta^*+\beta_I$. A computer algebra system confirms that $\mathcal{J}_{\mathcal{T}}$ has a simple zero eigenvalue at $\beta=\beta^{*}$, as required by Theorem~4.1 of \cite{castillo2004dynamical}.

\subsection{Eigenvectors of \texorpdfstring{$\mathcal{J}_{\mathcal{T}}$}{J}}
\label{subsubsec:CEE_eigenvectors}

Let $\mathbf{w}=(w_1,\ldots,w_{12})^{\top}$ be the right eigenvector and $\mathbf{v}=(v_1,\ldots,v_{12})$ the left eigenvector of $\mathcal{J}_{\mathcal{T}}$ corresponding to the zero eigenvalue, normalized so that $\mathbf{v}\cdot\mathbf{w}=1$. The derivation follows by solving $\mathcal{J}_{\mathcal{T}}\mathbf{w}=\mathbf{0}$ and $\mathbf{v}\mathcal{J}_{\mathcal{T}}=\mathbf{0}$ (see Appendix \ref{app:proof_CEE_eigenvectors}); the result is as follows.

With $w_3>0$ and $w_9>0$ as free positive parameters, the non-zero components of $\mathbf{w}$ are
\begin{equation}
  \begin{aligned}
    w_1 &= \frac{
      \gamma\bigl(\gamma_b\iota_b\kappa_S\,w_3
                 +\iota_b j_2\theta\,w_3
                 +\iota j_3\kappa_S\,w_9\bigr)
      - \gamma j_3 w_3(\beta_c j_2+\beta_I\kappa_S)
      - j_3 j_4 j_8\kappa_S\,w_9
    }{\gamma j_3\kappa_S\mu},\\[4pt]
    w_2 &= \frac{j_2\,w_3}{\kappa_S},\qquad
    w_4  = \frac{(\gamma_b\kappa_S+j_2\theta)\,w_3}{j_3\kappa_S},\qquad
    w_5  = \frac{j_8}{\gamma},\\[2pt]
    &w_6=w_7=w_8=w_{10}=w_{11}=w_{12}=0.
  \end{aligned}
  \label{eqn:right_eigenvector}
\end{equation}

With $v_3>0$ and $v_8>0$ as free positive parameters, the non-zero components of $\mathbf{v}$ are
\begin{equation}
  \begin{aligned}
    &v_1=v_4=v_9=v_{12}=0,\qquad v_3>0,\quad v_8>0,\\[4pt]
    v_2  &= -\frac{\kappa_S}{\beta_c-j_1}\,v_3,\qquad
    v_5   = \frac{j_7}{\iota_b+\beta^*}\,v_8,\\[4pt]
    v_{11} &= \frac{(\beta_c-j_1)\,\iota\,v_3
                   -\beta_I\kappa_S}{(\beta_c-j_1)\,j_{10}},\qquad
    v_{10}  = \frac{(\beta_c+\iota)\,v_2+\kappa_R\,v_{11}}{j_9},\\[4pt]
    v_7  &= \frac{\beta_I\,v_2+j_4\,v_5+\gamma_b\,v_8
                 +\gamma\,v_{11}}{j_6},\\[4pt]
    v_6  &= \frac{\beta_c\,v_2+j_4\,v_5+\kappa_I\,v_7
                 +\theta\,v_8+\gamma\,v_{10}}{j_5}.
  \end{aligned}
  \label{eqn:left_eigenvector}
\end{equation}

\begin{remark}
The component $v_2=-\kappa_S/(\beta_c-j_1)$ is the key quantity. When $\beta_c>j_1$ (the backward bifurcation regime) we have $v_2<0$; when $\beta_c<j_1$ we have $v_2>0$. This sign change drives the bifurcation direction, as shown in Section \ref{subsec:CEE_bifurcation}.
\end{remark}

\subsection{Bifurcation direction}
\label{subsec:CEE_bifurcation}

Theorem~4.1 of \cite{castillo2004dynamical} requires the two scalar coefficients
\begin{equation}
  a = \sum_{k,i,j=1}^{12}
      v_k\,w_i\,w_j\,
      \frac{\partial^2 y_k}{\partial x_i\,\partial x_j}
      \bigg|_{x_{ab}^{0}},
  \qquad
  b = \sum_{k,i=1}^{12}
      v_k\,w_i\,
      \frac{\partial^2 y_k}{\partial x_i\,\partial\beta^*}
      \bigg|_{x_{ab}^{0}}.
  \label{eqn:ab_definition}
\end{equation}
The bifurcation is backward (subcritical) if $a>0$, $b>0$; forward (supercritical) if $a<0$, $b>0$.

\paragraph{Coefficient $b$.}
The only non-zero contribution to $b$ comes from $y_5$ and $x_5$ (see Appendix \ref{app:Bif_Coeff}):
\begin{equation}
  b
  = v_5\,w_5\,
    \frac{\partial^2 y_5}{\partial x_5\,\partial\beta^*}
    \bigg|_{x_{ab}^{0}}
  = v_5\,w_5 = 1 > 0.
  \label{eqn:b_coeff}
\end{equation}
Since $b>0$ always, the bifurcation direction is determined solely by the sign of $a$.

\paragraph{Coefficient $a$: explicit formula.} Substituting the non-zero second partial derivatives from Appendix \ref{app:Bif_Coeff} into~\eqref{eqn:ab_definition} and discarding all zero contributions yields
\begin{equation}
  a
  = \underbrace{v_2\,w_1\!\left(w_2\beta_c+\beta_I\right)}_{\displaystyle\mathbf{[I]}}
  \underbrace{-\;w_5\,(w_2+1)\,v_2\,j_4}_{\displaystyle\mathbf{[II]}}
  +\underbrace{v_5\,w_5\!\left(w_1 j_4-w_2\beta_c-\beta_I\right)}_{\displaystyle\mathbf{[III]}}
  +\underbrace{w_5\,v_6\!\left(w_2(j_4+\beta_c)+\beta_I\right)}_{\displaystyle\mathbf{[IV]}}
  +\underbrace{w_5\,j_4\,(v_7+w_4)}_{\displaystyle\mathbf{[V]}}.
  \label{eqn:a_explicit}
\end{equation}
Each group has a clear biological role:
\begin{enumerate}[\roman*.]
  \item [\hspace{0pt}] Group \textbf{[I]} - bacterial secondary infections in the susceptible class: $w_2\beta_c$ (via carriers, weighted by mean carrier sojourn $j_2/\kappa_S$) plus $\beta_I$ (via active-disease individuals).
  \item [\hspace{0pt}] Group \textbf{[II]} - depletion of susceptibles through viral spread at rate $j_4$ along carrier ($w_2$) and direct ($+1$) pathways.
  \item [\hspace{0pt}] Group \textbf{[III]} - modulation of viral spread by the bacterial carrier dynamics through the co-infected susceptible class (weight $v_5$).
  \item [\hspace{0pt}] Group \textbf{[IV]} - co-infection amplification through the co-infected carrier class $x_{ic}$ (weight $v_6$).
  \item [\hspace{0pt}] Group \textbf{[V]} - coupling through disease progression ($v_7$, doubly-infected class) and bacterial recovery ($w_4$),  both scaled by the viral replication rate $j_4$.
\end{enumerate}

\paragraph{Sign of $a$: dominant-term analysis.} Numerical evaluation confirms that Group~\textbf{[I]} contributes more than $99\%$ of $|a|$ throughout the biologically plausible parameter space (see Table \ref{tab:ab_numerical} for representative values). Combining Groups~\textbf{[I]} and~\textbf{[II]}, which share the factor $v_2$:
\begin{equation}
  \mathbf{[I]}+\mathbf{[II]}
  = v_2\!\underbrace{\bigl[w_1(w_2\beta_c+\beta_I)-w_5(w_2+1)j_4\bigr]}_{\displaystyle\Psi}
  = -\frac{\kappa_S}{\beta_c-j_1}\,\Psi.
  \label{eqn:IpII}
\end{equation}
The quantity $\Psi$ depends on $w_1$, the DFE susceptible-pool weight. The numerator of $w_1$ (equation~\eqref{eqn:right_eigenvector}) contains the term $-\gamma j_3(\beta_c j_2+\beta_I\kappa_S)$, which is strictly negative whenever the parameters are positive. For all $\beta_c>j_1$ (the proposed backward bifurcation regime) this negative term dominates, giving $w_1<0$ and therefore $\Psi<0$. Substituting into~\eqref{eqn:IpII}:
\begin{equation}
  \beta_c>j_1
  \;\Longrightarrow\;
  \mathbf{[I]}+\mathbf{[II]}
  = -\frac{\kappa_S}{\underbrace{\beta_c-j_1}_{>0}}\,
    \underbrace{\Psi}_{<0} > 0
  \;\Longrightarrow\; a > 0.
  \label{eqn:a_positive}
\end{equation}
The analogous argument for $\beta_c<j_1$ gives $v_2>0$, $\mathbf{[I]}+\mathbf{[II]}<0$, and $a<0$. Hence
\begin{equation}
  \operatorname{sgn}(a)
  = \operatorname{sgn}(\beta_c - j_1)
  = \operatorname{sgn}(\beta_c - \kappa_S - \theta - \mu).
  \label{eqn:a_sign}
\end{equation}

\begin{table}[pos=ht]
  \centering
  \caption{Numerical decomposition of $a$ for two representative
           parameter sets.
           Baseline: $\gamma=0.33$, $\gamma_b=0.1$, 
           $\beta_I=0.4$,
           $\kappa_S=0.005$, $\kappa_I=0.05$,
           $\kappa_R=0.02$, 
           $\theta=0.032$,
           $\iota=0.00$, 
           $\iota_b=0.005$, $\mu=0.0001$.
           The threshold is $j_1=\kappa_S+\theta+\mu=0.0371$.}
  \label{tab:ab_numerical}
  \renewcommand{\arraystretch}{1.1}
  \begin{tabular}{@{}lrrrrrrr@{}}
    \hline
    Case & $\beta_c$ & \textbf{[I]} & \textbf{[II]} & \textbf{[III]}
         & \textbf{[IV]} & \textbf{[V]} & $a$ \\
    \hline
    Forward ($\beta_c<j_1$) & 0.020 & $-174.0809$ & $-0.0006$ & $-0.5942$ & $+0.0176$ & $+0.0153$ & $-174.6428$ \\
    Backward ($\beta_c>j_1$) & 0.050& $+3665.3115$ & $+0.0008$ & $-5.3776$ & $+0.0151$ & $+0.0150$ & $+3659.9646$ \\ 
    \hline
  \end{tabular}
\end{table}
 
\paragraph{Biological interpretation of the criterion.} The rate $j_1=\kappa_S+\theta+\mu$ is the total per-capita rate at which a carrier leaves the carrier class: by progressing to active disease (rate $\kappa_S$), clearing spontaneously (rate $\theta$), or dying (rate $\mu$). The condition $\beta_c>j_1$ states that a carrier produces more than one new carrier per unit time \emph{before} leaving the class — sustaining a self-reinforcing carrier reservoir that persists even when $\mathcal{R}_0<1$. This is precisely the same mechanism that drives the backward bifurcation at the QVFE (Theorem \ref{thm:ReducedModel_NoVirus_Backward}), confirming that the bifurcation structure of the full co-infection model is inherited from the bacterial-only subsystem.

\begin{lemma}[Bifurcation type at the CEE] \label{lem:Backward_Bifurcation} Consider System~\eqref{eqn:MainModel} with $\mathcal{R}_{0b}^{T1}<\mathcal{R}_{0a}^{V1}$ (so $\mathcal{R}_0=\mathcal{R}_{0a}^{V1}$) and bifurcation parameter $\beta^{*}=j_4$. Let $a$ be defined by~\eqref{eqn:a_explicit} and $b=1>0$ by~\eqref{eqn:b_coeff}.

\begin{enumerate}[\roman*]
  \item \textbf{Backward bifurcation.} If $\beta_c>\kappa_S+\theta+\mu$, then $a>0$ and the system undergoes a subcritical bifurcation at $x_{ab}^{0}$ when $\mathcal{R}_0=1$. A stable CEE coexists with the stable DFE for $\mathcal{R}_0$ slightly below~$1$.
  
  \item \textbf{Forward bifurcation.} If $\beta_c\leq\kappa_S+\theta+\mu$, then $a\leq 0$ and the bifurcation is supercritical. The DFE is the only stable equilibrium for $\mathcal{R}_0<1$, and the unique CEE $x_{ab}^{*}$ is locally asymptotically stable for $\mathcal{R}_0>1$ near~$1$.
\end{enumerate}
\end{lemma}

\begin{proof}
The proof is an application of Theorem~4.1 of \cite{castillo2004dynamical}. With $b=1>0$ established in \ref{eqn:b_coeff} and the sign of $a$ determined by \ref{eqn:a_sign}- \ref{eqn:a_positive}, the two cases of the theorem apply directly. The supporting eigenvector calculations are in Appendix \ref{app:proof_CEE_eigenvectors}, and the second partial derivatives used in \ref{eqn:a_explicit} are listed in Appendix \ref{app:Bif_Coeff}.
\end{proof}

\begin{remark}[Epidemiological consequences]
\label{rem:CEE_backward}
Lemma \ref{lem:Backward_Bifurcation} has three concrete implications.
\begin{enumerate}[\roman*]
  \item \textbf{Threshold reduction is necessary but not sufficient.} When $\beta_c>j_1$, reducing $\mathcal{R}_0$ below~$1$ by any means (e.g.\ reducing viral transmission $\beta$, or increasing viral recovery $\gamma$) does not eliminate co-infection. The coexistence equilibrium persists as a stable attractor, and eradication requires pushing the system through the  saddle-node bifurcation, demanding a larger sustained effort.

  \item \textbf{Spontaneous clearance $\theta$ is the critical lever.} The backward bifurcation is removed as soon as $\theta>\beta_c-\kappa_S-\mu$. Interventions that promote carrier clearance — such as prophylactic treatment or nutritional supplementation in pneumoccocal settings — can shift the system from subcritical to supercritical, restoring the guarantee that $\mathcal{R}_0<1$ implies elimination.

  \item \textbf{Structural consistency across equilibrium levels.} The same criterion $\beta_c>j_1$ governs the backward bifurcation at the QVFE (single bacterial disease, Theorem \ref{thm:ReducedModel_NoVirus_Backward}) and at the CEE (both diseases coexisting). This consistency confirms that the bifurcation is a property of the carrier transmission mechanism, not an artefact of the single-disease reduction.
\end{enumerate}
\end{remark}

\section{Sensitivity Analysis}
\label{sec:Sensitivity}

While the analytical equations calculated in the previous sections outline the general structure of the co-infection system, they do not tell us which parameters are the most critical drivers of the epidemic thresholds and invasion conditions when biological uncertainties are involved. In earlier single-disease work, we used normalized forward-sensitivity indices, evaluated locally at a fixed parameter set, to rank the drivers of an epidemic threshold \cite{fome2023deterministic}. That approach is less well suited here: several of the present outputs ($x_{is}^{bf}$, $x_{sc}^{vf}$) are defined only conditionally, existing solely once their respective reproduction number exceeds one, and the local derivative underlying
a forward-sensitivity index cannot be computed at such a threshold boundary; more generally, a single-point local index cannot capture how influential a parameter is across the full biologically plausible
range considered in Table \ref{tab:parameter_justification}. To bridge this gap and understand how variation in each biological parameter impacts the system globally, we instead perform a global
sensitivity analysis across the seven key model outputs: $\mathcal{R}_{0a}^{V1}$, $\mathcal{R}_{0b}^{T1}$, $\mathcal{R}_0$, $\mathbf{Inv}^{\mathbf{B}}$, $\mathbf{Inv}^{\mathbf{V}}$, $x_{is}^{bf}$, and $x_{sc}^{vf}$. This is achieved using Latin Hypercube Sampling (LHS) paired with Partial Rank Correlation Coefficients (PRCC), a standard and robust approach for evaluating nonlinear, monotone disease transmission models \cite{blower1994sensitivity,marino2008methodology}. The full computational pipeline and algorithmic implementation are detailed in Appendix \ref{app:LHS_PRCC}.

\begin{figure}[pos=ht]
  \centering
\includegraphics[width=\textwidth]{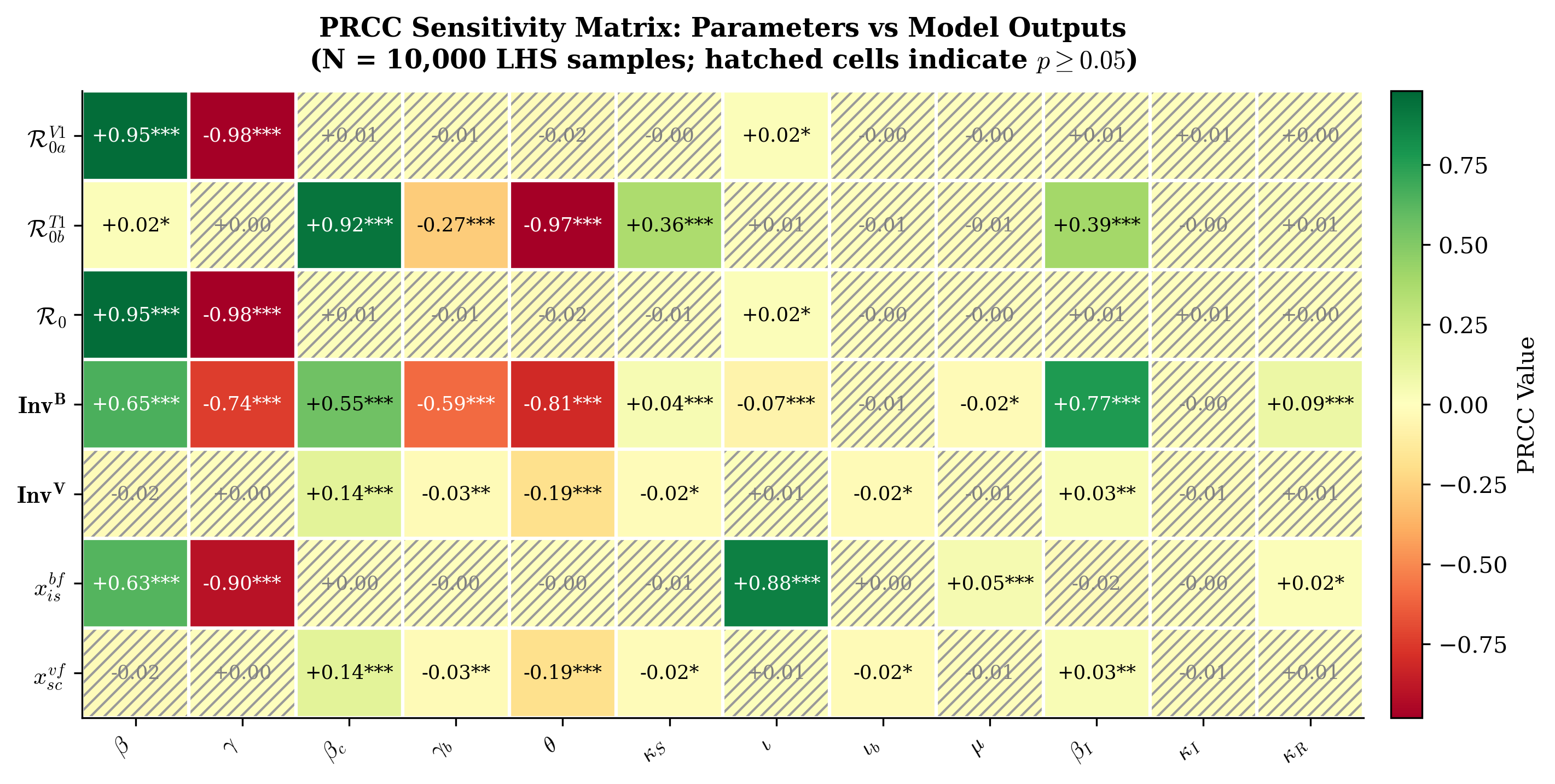}
  \caption{PRCC heatmap showing the sensitivity of all seven model outputs to the twelve model parameters. Green indicates a positive correlation, red indicates a negative correlation, and hatched cells denote correlations that are not statistically significant ($p\geq0.05$). PRCC values are annotated with significance stars ($^{*}p<0.05$, $^{**}p<0.01$, $^{***}p<0.001$) based on $N=5{,}000$ Latin Hypercube samples.}
  \label{fig:prcc_heatmap}
\end{figure}

Figure \ref{fig:prcc_heatmap} shows that $\mathcal{R}_0$ has the same PRCC signature as $\mathcal{R}_{0a}^{V1}$, driven by $\beta$ (positive) and $\gamma$ (negative), consistent with $\mathcal{R}_0=\mathcal{R}_{0a}^{V1}$ at these parameters. $\mathcal{R}_{0b}^{T1}$ is instead driven by $\beta_c$ (positive) and $\theta$ (negative), with $\beta_I$ and $\gamma_b$ contributing more weakly -- confirming the two thresholds depend on largely disjoint parameter sets. By contrast, $\mathbf{Inv}^{\mathbf{B}}$ loads significantly on the purely viral parameters $\beta$ and $\gamma$, the expected signature of the co-infection amplification effect (Section \ref{subsec:BacteriaEmpire}): higher influenza burden increases the size of $x_{is}^{bf}$, the sub-population in which bacterial transmission is elevated. $\mathbf{Inv}^{\mathbf{V}}$ shares its PRCC profile with $x_{sc}^{vf}$, and $\iota$ is the dominant driver of $x_{is}^{bf}$.

\section{Discussion}\label{sec:discussion}

This study asks whether giving a co-infection model explicit carriage-stage structure -- rather than folding bacterial status into severity classes or a single generic carrier compartment -- changes the qualitative epidemic behavior of an influenza--pneumococcus system. Building a twelve-compartment model that tracks each host's joint viral and bacterial history through recovery, we show that it does: the model reproduces the well-documented asymmetry where influenza strongly promotes pneumococcal disease but not the reverse, and it isolates a precise, testable condition -- carrier transmission rate exceeding the carrier's total exit rate, $\beta_c>\kappa_S+\theta+\mu$ -- under which the system undergoes a backward bifurcation, meaning reducing $\mathcal{R}_0$ below one is not enough to guarantee elimination. A global sensitivity analysis then identifies which parameters actually drive each threshold, showing that the bacterial and viral reproduction numbers are governed by almost entirely separate parameter sets, while bacterial \emph{invasion} potential is strongly shaped by viral parameters -- the quantitative signature of the synergy the model was built to capture.

These results carry direct weight for sub-Saharan Africa, the region motivating this study, and not merely as structural interest. Surveillance across the region shows both pathogens circulating at levels that make their interaction a routine clinical event rather than a rare coincidence. National sentinel surveillance in Tanzania documents sustained, year-round influenza circulation with marked seasonal peaks \cite{mmbaga2012results,shedura2023evaluation}, and in Soweto, influenza and RSV together accounted for a substantial share of adult deaths \cite{karstaedt2009influenza}, with pneumonia remaining a leading cause of adult death in Kenya \cite{scott2000aetiology}. On the bacterial side, nasopharyngeal carriage of \emph{S. pneumoniae} is common and persistent across East Africa: a Dar es Salaam cohort found high carriage prevalence, including penicillin-non-susceptible strains, among young children \cite{moyo2012penicillin}, and northern Tanzanian surveillance confirms that carriage remains substantial even after PCV13 introduction, with continued co-occurrence of vaccine and non-vaccine serotypes \cite{emgaard2019carriage,emgaard2024co}. A mechanism that keeps this carriage transmissible even as $\mathcal{R}_0$ falls below one is therefore not a mathematical curiosity but a candidate explanation for why reducing transmission alone has not been sufficient to eliminate disease in these settings.

Evidence that the two pathogens are causally linked, not merely co-occurring, is itself substantially African in origin. In a large randomized trial in South Africa, vaccinating infants against pneumococcus prevented 31\% of pneumonia cases linked to any of seven respiratory viruses \cite{madhi2004role} -- direct evidence that reducing one pathogen changes the burden of the other. That a purely bacterial intervention reduced viral-associated disease is hard to reconcile with a model that has no explicit bacterial compartment for the virus to interact with, and it is consistent with the amplification pathway formalized here by $\mathbf{Inv}^{\mathbf{B}}$ (Section~\ref{subsec:BacteriaEmpire}): removing the bacterial reservoir removes the amplification the virus would otherwise exploit. This pattern recurs across the continent at the level of routine surveillance: co-detection of influenza and \emph{S. pneumoniae} is common among children hospitalized with respiratory illness in South Africa and Senegal \cite{moleleki2025respiratory,knobbe2019pathogens}, and a further South African study found tuberculosis, pneumococcal disease, and influenza hospitalizations rising and falling together each year \cite{dangor2014temporal}. Tanzanian data add a diagnostic dimension to the same picture: in a large outpatient study, most febrile children did not have malaria, and a substantial share of the remaining fevers had no single pathogen identified \cite{d2014beyond}, while children hospitalized with respiratory illness in Tanzania frequently show viral and bacterial pathogens detected together rather than as a single cause \cite{ngocho2020viral,kelly2023etiologies}. None of these studies can, by itself, distinguish a shared seasonal driver from a genuine biological interaction, but the mechanism identified in Section~\ref{sec:CEE} gives a concrete, testable pathway -- immune compromise during influenza elevating carrier-to-disease progression -- through which such covariation would be expected to arise even without a common external cause.

This carriage-driven mechanism is only visible because of how the present model is built, which becomes clear when set against existing influenza--pneumococcus models, including African ones built for a related purpose. Kenyan dynamic transmission models of pneumococcal carriage, fitted to the Kilifi longitudinal cohort, have shown that residual vaccine-serotype carriage can persist for years after PCV catch-up campaigns even as invasive disease declines \cite{ojal2017sustained,abdullahi2012rates}; a structurally similar finding for Malawi \cite{lourencco2019determinants} shows this is not specific to one country's rollout. These models, however, treat pneumococcus in isolation and cannot address whether an independent respiratory virus could sustain or accelerate that residual carriage -- the question the present model was built to answer. As discussed in Section~\ref{sec:intro}, Ali et al. \cite{ali2025model} also report a backward bifurcation in an influenza--pneumococcus co-infection model, at $\mathcal{R}_0=\max(\mathcal{R}_f,\mathcal{R}_p)=1$, but give no explicit condition on carrier transmission. The present analysis shows that the bifurcation can instead be traced to a single interpretable condition on the carrier compartment, $\beta_c>\kappa_S+\theta+\mu$ (Section~\ref{sec:CEE}), and that this same condition governs both the single-disease bacterial endemic state and the full coexistence state (Section~\ref{subsubsec:VirusEmpire}). P\'erez and Oluyori \cite{perez2022model} similarly track joint infection history through recovery, but classify bacterial infection by severity rather than by carriage stage, so their model has no asymptomatic carrier compartment at all; since the backward bifurcation found here comes specifically from transmission by carriers (Section~\ref{sec:CEE}), a model without such a compartment cannot produce this result under any parameter choice. Taken together, these comparisons indicate that the results reported here follow from the carriage-stage structure itself, rather than from a particular calibrated parameter set that a differently-structured model could also happen to reach.

The same pattern -- backward bifurcation confined to the bacterial sub-model and the coexistence state, with the viral sub-model remaining forward (Theorem~\ref{thm:ReducedModel_NoBacterial_4ward}) -- has also been reported for other co-infection pairs, though for a different underlying reason. Bhunu et al. \cite{bhunu2009modeling} found, for HIV--tuberculosis co-infection, that the tuberculosis sub-model alone can produce a backward bifurcation while the HIV sub-model remains globally stable, and the same qualitative asymmetry was later reproduced for COVID-19--tuberculosis co-infection \cite{mekonen2022mathematical,omame2021fractional}. There, the bifurcation is typically traced to a delay in progression from latent to active tuberculosis -- itself a disease pattern with substantial sub-Saharan African burden, underscoring that this class of threshold behavior is not a mathematical edge case in the settings this paper is concerned with. Here, by contrast, the bifurcation condition (Equation~\eqref{eqn:a_sign}) is entirely a statement about direct transmission from the \emph{asymptomatic carrier} class relative to the rate at which carriers leave that class. This distinction matters for intervention design: where a latency-delay mechanism points toward shortening time-to-diagnosis, a carrier-transmission mechanism points toward reducing $\beta_c$ directly -- for example, through measures that lower contact-adjusted shedding from asymptomatic carriers -- or raising the spontaneous clearance rate $\theta$ (Remark~\ref{rem:CEE_backward}(ii)).

This carrier-transmission mechanism is also distinct from what we have found in related modeling frameworks that treat the two pathogens as a virus--virus pair rather than a virus--bacterium pair. The backward bifurcation identified in Sections~\ref{subsec:QDFE} and~\ref{sec:CEE} does not arise in the purely viral--viral cross-immunity models we have previously used for influenza and SARS-CoV-2 co-circulation under quarantine \cite{fome2025analysis,fome2025interplay}, nor in the strain-competition setting studied in \cite{fome2026dynamics}, where full cross-immunity instead drives competitive exclusion of the strain with the lower quarantine reproduction number and only partial cross-immunity permits coexistence. It is, however, consistent with -- and gives fuller population-level structure to -- our earlier finding, in a simpler two-disease single-host model, that backward bifurcation emerges once each sub-strain's reproduction number exceeds a threshold value driven specifically by asymptomatic transmission \cite{fome2024nordic}. The present analysis identifies the precise mechanism behind that earlier observation and shows that the same condition, rather than being specific to that simplified two-disease setting, persists once the model is extended to a full twelve-compartment influenza--pneumococcus system with carriage-stage structure and joint infection history. These qualitative comparisons are complemented by the PRCC results (Figure~\ref{fig:prcc_heatmap}), which give a population-level, quantitative signature of the same synergy. The clean separation between $\mathcal{R}_{0a}^{V1}$, driven by $\beta$ and $\gamma$ alone, and $\mathcal{R}_{0b}^{T1}$, driven by $\beta_c$ and $\theta$ with $\beta_I$ and $\gamma_b$ playing a secondary role, confirms that the two single-disease thresholds are structurally decoupled, exactly as built into the model. The fact that $\mathbf{Inv}^{\mathbf{B}}$ is nonetheless strongly loaded on the purely viral parameters $\beta$ and $\gamma$ is therefore not a modeling artefact but the specific, testable signature of viral--bacterial synergy that the invasion-number construction (Section~\ref{subsec:BacteriaEmpire}) was designed to expose.

This backward-bifurcation mechanism is not merely theoretical in the region motivating this study -- it maps closely onto a documented, continent-wide vaccination experience. Sub-Saharan Africa has now largely completed pneumococcal conjugate vaccine rollout -- as of the early 2020s, all but six of Africa's 54 countries had introduced PCV into national or subnational immunization programmes, most between 2009 and 2016 \cite{dhaliwal2024introduction} -- yet vaccination has not behaved as it did in high-income settings. A seven-year surveillance study in Malawi found that vaccine-serotype carriage remained substantially elevated long after PCV13 introduction, declining far more slowly than expected and persisting at close to 17\% among vaccinated children \cite{swarthout2020high}, a finding corroborated by an independent Malawian modeling study attributing the residual carriage to incomplete competitive displacement of vaccine serotypes \cite{lourencco2019determinants}. In Kenya, a transmission model fitted to pre-vaccination Kilifi data likewise projected that vaccine-serotype carriage would persist well beyond the level needed for elimination, even a decade after catch-up vaccination \cite{ojal2017sustained}. Tanzanian carriage surveillance shows the same qualitative pattern in the PCV13 era, with continued detection of vaccine and non-vaccine serotypes together in young children \cite{emgaard2019carriage,emgaard2024co}. This is precisely the qualitative pattern the model predicts when carrier transmission remains above its exit rate even as $\mathcal{R}_0$ falls: vaccination lowers the reproduction number, but a persistent carrier reservoir can survive the transition, consistent with Theorem~\ref{thm:ReducedModel_NoVirus_Backward} and Remark~\ref{rem:backward_bif}. The present model does not include vaccination explicitly (Section~\ref{subsec:limitations}), so this comparison is qualitative rather than a direct test of the theory; still, the correspondence between four independent African surveillance and modeling efforts and a mechanism the present model isolates analytically strengthens the case that carrier transmission, not just the reproduction number, deserves attention as an independent target in these settings.

\subsection{Limitations}
\label{subsec:limitations}
The model and its analysis carry several limitations. All stability results are local, and the presence of a backward bifurcation means local stability of the disease-free equilibrium does not guarantee global elimination. The model is deterministic, considers a single influenza strain and a single pneumococcal serotype, and does not include vaccination or treatment -- an omission that matters specifically for the African rollout comparison drawn above, since none of the Kenyan, Malawian, or Tanzanian studies cited there can be fitted directly without a vaccination compartment. Three co-infection-specific parameters ($\beta_I,\kappa_I,\kappa_R$) are set from the qualitative direction of mechanistic evidence rather than from measured values (Section~\ref{subsec:ParameterJustification}), consistent with a parameterization gap documented across the field \cite{dhakal2025mechanistic}; the model has not been fitted to surveillance data, so all quantitative outputs should be read as statements about qualitative structure rather than forecasts. Finally, the model distinguishes only current, recovered, and never-infected viral status, so it does not resolve the specific 4--6 day window of heightened risk identified in within-host studies \cite{shrestha2013timedose}, and cannot capture within-window timing effects on invasion risk.

\subsection{Conclusion}
\label{sec:conclusion}
This paper introduced a twelve-compartment influenza--pneumococcus co-infection model that tracks carriage stage jointly with viral history through recovery, and showed that this structure produces a backward bifurcation governed by a single interpretable condition, $\beta_c>\kappa_S+\theta+\mu$, under which reducing $\mathcal{R}_0$ below one is not sufficient to eliminate co-infection. A global sensitivity analysis further showed that the viral and bacterial reproduction thresholds are driven by largely separate parameters, while bacterial invasion potential is shaped by viral transmission and recovery rates, giving a quantitative signature of the synergy between the two pathogens. Set against the residual vaccine-serotype carriage documented in Kenya, Malawi, and Tanzania after PCV introduction, these results indicate that interventions targeting carriage -- raising spontaneous clearance or reducing carrier transmission -- are structural control levers in their own right, not merely proxies for reducing invasive disease, in the sub-Saharan African settings this paper has focused on. They motivate fitting the model to co-infection surveillance data from the region and extending it to include vaccination, treatment, and multiple strains or serotypes.

\section*{Appendices}\label{sec:appendices}

\appendix

\section{Proof of Lemma \ref{lem:Positivity}}\label{app:Positivity}
\begin{proof}
We apply the standard \emph{invariance principle for cooperative systems} \cite{smith2008monotone}: $\mathbb{R}^{12}_{+}$ is positively invariant under \eqref{eqn:MainModel} if and only if, whenever a component $X_{ab}=0$ and all others are non-negative, $X_{ab}'\geq 0$. Setting each $X_{ab}=0$ in turn removes the terms proportional to $X_{ab}$ itself, leaving only inflows from other compartments:
\begin{align*}
  X_{SS}'|_{X_{SS}=0} &= \Lambda+\iota X_{RS}+\iota_b X_{SR}, &
  X_{IS}'|_{X_{IS}=0} &= X_{SS}\xi_V+\iota_b X_{IR}, \\
  X_{SC}'|_{X_{SC}=0} &= X_{SS}\xi_B+\iota X_{RC}, &
  X_{IC}'|_{X_{IC}=0} &= X_{SC}\xi_V+X_{IS}\xi_B, \\
  X_{SI}'|_{X_{SI}=0} &= \kappa_S X_{SC}+\iota X_{RI}, &
  X_{II}'|_{X_{II}=0} &= X_{SI}\xi_V+\kappa_I X_{IC}, \\
  X_{SR}'|_{X_{SR}=0} &= \gamma_b X_{SI}+\iota X_{RR}+\theta X_{SC}, &
  X_{IR}'|_{X_{IR}=0} &= X_{SR}\xi_V+\gamma_b X_{II}+\theta X_{IC}, \\
  X_{RS}'|_{X_{RS}=0} &= \gamma X_{IS}+\iota_b X_{RR}, &
  X_{RI}'|_{X_{RI}=0} &= \kappa_R X_{RC}+\gamma X_{II}, \\
  X_{RC}'|_{X_{RC}=0} &= X_{RS}\xi_B+\gamma X_{IC}, &
  X_{RR}'|_{X_{RR}=0} &= \gamma_b X_{RI}+\gamma X_{IR}+\theta X_{RC}.
\end{align*}
Since $\xi_V,\xi_B\geq 0$ (non-negative combinations of non-negative states divided by $N>0$) and all rate parameters are non-negative, every expression above is non-negative. Hence $X_{ab}'|_{X_{ab}=0}\geq0$ for all twelve compartments, establishing \eqref{eqn:WP_positivity_statement}.
\end{proof}

\section{Derivation of the Next-Generation Matrices}
\label{app:RepNum}

Ordering the infected compartments as
\[
  \mathbf{x}_{\mathrm{inf}}
  = (x_{sc},\,x_{si},\,x_{is},\,x_{ic},\,x_{ii},\,x_{ir},\,x_{rc},\,x_{ri})^{\top},
\]
decomposing~\eqref{eqn:MainModel} into new-infection terms $\mathcal{F}$ and transition terms $\mathcal{V}$, and linearizing about the DFE~\eqref{eqn:DFE} gives the $8\times8$ matrices
\begin{equation}
  F =
  \begin{pmatrix}
    \beta_c & \beta_I & 0      & \beta_c & \beta_I & 0      & \beta_c & \beta_I \\
    0       & 0       & 0      & 0       & 0       & 0      & 0       & 0       \\
    0       & 0       & \beta  & \beta   & \beta   & \beta  & 0       & 0       \\
    0       & 0       & 0      & 0       & 0       & 0      & 0       & 0       \\
    0       & 0       & 0      & 0       & 0       & 0      & 0       & 0       \\
    0       & 0       & 0      & 0       & 0       & 0      & 0       & 0       \\
    0       & 0       & 0      & 0       & 0       & 0      & 0       & 0       \\
    0       & 0       & 0      & 0       & 0       & 0      & 0       & 0
  \end{pmatrix},
  \qquad
  V =
  \begin{pmatrix}
    j_1        & 0     & 0     & 0     & 0     & 0          & -\iota   & 0      \\
    -\kappa_S  & j_2   & 0     & 0     & 0     & 0          & 0        & -\iota \\
    0          & 0     & j_4   & 0     & 0     & -\iota_b   & 0        & 0      \\
    0          & 0     & 0     & j_5   & 0     & 0          & 0        & 0      \\
    0          & 0     & 0     & -\kappa_I & j_6 & 0         & 0        & 0      \\
    0          & 0     & 0     & -\theta & -\gamma_b & j_7   & 0        & 0      \\
    0          & 0     & 0     & -\gamma & 0    & 0          & j_9      & 0      \\
    0          & 0     & 0     & 0     & -\gamma & 0         & -\kappa_R & j_{10}
  \end{pmatrix}.
  \label{eqn:F_matrix}
\end{equation}
Computing $V^{-1}$ and forming $FV^{-1}$ yields the next-generation matrix~\eqref{eqn:FV_Inverse}, whose non-trivial entries are, by disease:

\paragraph{Bacterial.}
\begin{align}
  \mathcal{R}_{0b}^{T1} &= \frac{\beta_c j_2 + \beta_I\kappa_S}{j_1 j_2},
  \label{eqn:R0bT1}\\
  \mathcal{R}_{0b}^{I1} &= \frac{\beta_I}{j_2},
  \label{eqn:R0bI1}\\
  \mathcal{R}_{0b}^{I2} &= \frac{\beta_I}{j_5}\cdot
    \frac{\kappa_S\iota\gamma j_6 j_{10}+\kappa_I(\iota\gamma+j_2\gamma+j_2 j_{10})j_1 j_9+\kappa_R(\iota\gamma+j_2\gamma)j_1 j_6}{j_1 j_2 j_9 j_{10}},
  \label{eqn:R0bI2}\\
  \mathcal{R}_{0b}^{I3} &= \frac{\beta_I}{j_6}\cdot\frac{j_2 j_{10}+\gamma j_2+\gamma\iota}{j_2 j_{10}},
  \label{eqn:R0bI3}\\
  \mathcal{R}_{0b}^{I4} &= \beta_I\cdot\frac{j_1(\kappa_R+j_9)(\iota+j_2)+\kappa_S j_{10}\iota}{j_1 j_2 j_9 j_{10}},
  \label{eqn:R0bI4}\\
  \mathcal{R}_{0b}^{I5} &= \frac{\beta_I}{j_{10}}\cdot\frac{j_2+\iota}{j_2},
  \label{eqn:R0bI5}\\
  \mathcal{R}_{0b}^{C1} &= \frac{\beta_c}{j_9}\cdot\frac{j_1+\iota}{j_1},
  \label{eqn:R0bC1}\\
  \mathcal{R}_{0b}^{C2} &= \frac{\beta_c}{j_5}\cdot\frac{j_1 j_9+\gamma j_1+\gamma\iota}{j_1 j_9}.
  \label{eqn:R0bC2}
\end{align}

\paragraph{Viral.}
\begin{align}
  \mathcal{R}_{0a}^{V1} &= \frac{\beta}{j_4},
  \label{eqn:R0aV1}\\
  \mathcal{R}_{0a}^{V2} &= \beta\cdot\frac{\kappa_I(j_4\gamma_b+\iota_b\gamma_b+j_4 j_7)+\theta j_6(j_4+\iota_b)+j_4 j_6 j_7}{j_4 j_5 j_6 j_7},
  \label{eqn:R0aV2}\\
  \mathcal{R}_{0a}^{V3} &= \frac{\beta}{j_6}\cdot\frac{j_4 j_7+\gamma_b j_4+\gamma_b\iota_b}{j_4 j_7},
  \label{eqn:R0aV3}\\
  \mathcal{R}_{0a}^{V4} &= \frac{\beta}{j_7}\cdot\frac{j_4+\iota_b}{j_4}.
  \label{eqn:R0aV4}
\end{align}

\section{Proof of Lemma \ref{lem:DFE_LAS}}\label{app:Derivation of lem:DFE_LAS}

\begin{proof}
Local asymptotic stability is equivalent to all eigenvalues of the
Jacobian evaluated at the equilibrium having strictly negative real
part \cite{perko2013differential,murray2002mathematical}.
The Jacobian of System~\eqref{eqn:MainModel} at $x_{ab}^{0}$ is the
$(12\times12)$ matrix
\begin{equation}
  J_{0}=
  \setcounter{MaxMatrixCols}{12}
  \begin{pmatrix}
    -\mu     & -\beta_c & -\beta_I & \iota_b  & -\beta   & -j_{12}  & -j_{13}       & -\beta   & \iota    & -\beta_c & -\beta_I & 0        \\
    0        & \beta_c-j_1 & \beta_I & 0       & 0        & \beta_c  & \beta_I       & 0        & 0        & \beta_c+\iota & \beta_I & 0   \\
    0        & \kappa_S & -j_2     & 0        & 0        & 0        & 0             & 0        & 0        & 0        & \iota    & 0        \\
    0        & \theta   & \gamma_b & -j_3     & 0        & 0        & 0             & 0        & 0        & 0        & 0        & \iota    \\
    0        & 0        & 0        & 0        & \beta-j_4 & \beta   & \beta         & \beta+\iota_b & 0   & 0        & 0        & 0        \\
    0        & 0        & 0        & 0        & 0        & -j_5     & 0             & 0        & 0        & 0        & 0        & 0        \\
    0        & 0        & 0        & 0        & 0        & \kappa_I & -j_6          & 0        & 0        & 0        & 0        & 0        \\
    0        & 0        & 0        & 0        & 0        & \theta   & \gamma_b      & -j_7     & 0        & 0        & 0        & 0        \\
    0        & 0        & 0        & 0        & \gamma   & 0        & 0             & 0        & -j_8     & 0        & 0        & \iota_b  \\
    0        & 0        & 0        & 0        & 0        & \gamma   & 0             & 0        & 0        & -j_9     & 0        & 0        \\
    0        & 0        & 0        & 0        & 0        & 0        & \gamma        & 0        & 0        & \kappa_R & -j_{10}  & 0        \\
    0        & 0        & 0        & 0        & 0        & 0        & 0             & \gamma   & 0        & \theta   & \gamma_b & -j_{11}
  \end{pmatrix},
  \label{eqn:Jacobian_Matrix}
\end{equation}
with $j_{12}=\beta+\beta_c$ and $j_{13}=\beta+\beta_I$.

\medskip
\noindent\textbf{Step~1: Ten eigenvalues by cofactor expansion.}

Exploiting the sparse structure of $J_0$, cofactor expansion along columns with a single non-zero entry yields ten eigenvalues directly:
\begin{equation}
  \lambda_1=-\mu,\quad
  \lambda_2=-j_3,\quad
  \lambda_3=-j_8,\quad
  \lambda_4=\beta-j_4,\quad
  \lambda_5=-j_{11},\quad
  \lambda_6=-j_7,
  \label{eqn:eigs_1to6}
\end{equation}
\begin{equation}
  \lambda_7=-j_{10},\quad
  \lambda_8=-j_9,\quad
  \lambda_9=-j_6,\quad
  \lambda_{10}=-j_5.
  \label{eqn:eigs_7to10}
\end{equation}

\noindent\textit{Eigenvalues $\lambda_1$--$\lambda_3$, $\lambda_5$--$\lambda_{10}$.} Each equals the negative of a sum of non-negative parameters containing $\mu>0$, hence is strictly negative:
\begin{alignat*}{3}
  \lambda_1 &= -\mu < 0, &\qquad
  \lambda_2 &= -(\iota_b+\mu)<0, &\qquad
  \lambda_3 &= -(\iota+\mu)<0,\\
  \lambda_5 &= -(\iota+\iota_b+\mu)<0, &\qquad
  \lambda_6 &= -(\gamma+\iota_b+\mu)<0, &\qquad
  \lambda_7 &= -(\gamma_b+\iota+\mu)<0,\\
  \lambda_8 &= -(\kappa_R+\theta+\iota+\mu)<0, &\qquad
  \lambda_9 &= -(\gamma+\gamma_b+\mu)<0, &\qquad
  \lambda_{10} &= -(\kappa_I+\gamma+\theta+\mu)<0.
\end{alignat*}

\noindent\textit{Eigenvalue $\lambda_4$.} Recalling $j_4=\gamma+\mu$,
\begin{equation}
  \lambda_4
  = \beta - j_4
  = j_4\!\left(\frac{\beta}{j_4}-1\right)
  = j_4\!\left(\mathcal{R}_{0a}^{V1}-1\right).
  \label{eqn:lambda4}
\end{equation}
Since $j_4>0$, we have $\lambda_4<0$ if and only if $\mathcal{R}_{0a}^{V1}<1$.

\medskip
\noindent\textbf{Step~2: The remaining $2\times2$ block.}

After extracting $\lambda_1,\ldots,\lambda_{10}$, the residual characteristic equation factors to
\begin{equation}
  \mathcal{J}_{2}
  = \begin{pmatrix}   
      \beta_c-j_1 & \beta_I \\
      \kappa_S    & -j_2
    \end{pmatrix},
  \label{eqn:J2x2}
\end{equation}
with eigenvalues $\lambda_{11}$ and $\lambda_{12}$. By the Routh--Hurwitz criterion \cite{murray2002mathematical}, both roots have strictly negative real part if and only if $\operatorname{tr}(\mathcal{J}_{2})<0$ and  $\det(\mathcal{J}_{2})>0$.

\noindent\textit{Trace.}
\[
  \operatorname{tr}(\mathcal{J}_{2})
  = \beta_c - j_1 - j_2.
\]
If $\mathcal{R}_{0b}^{T1}<1$, then $\beta_c j_2+\beta_I\kappa_S<j_1j_2$,
which implies $\beta_c<j_1$, so
$\operatorname{tr}(\mathcal{J}_{2})=\beta_c-j_1-j_2<-j_2<0$.

\noindent\textit{Determinant.}
\begin{equation}
  \det(\mathcal{J}_{2})
  = j_1 j_2 - \beta_c j_2 - \beta_I\kappa_S
  = j_1 j_2\!\left(1-\mathcal{R}_{0b}^{T1}\right).
  \label{eqn:det_J2x2}
\end{equation}
Since $j_1,j_2>0$, we have $\det(\mathcal{J}_{2})>0$ if and only if
$\mathcal{R}_{0b}^{T1}<1$.

\medskip
\noindent\textbf{Step~3: Conclusion.}

If $\mathcal{R}_{0a}^{V1}<1$ and $\mathcal{R}_{0b}^{T1}<1$, then all twelve eigenvalues satisfy $\operatorname{Re}(\lambda_i)<0$, so $x_{ab}^{0}$ is LAS by linearization theory \cite{perko2013differential}.

For the converse: $\mathcal{R}_{0a}^{V1}>1$ gives $\lambda_4>0$ by~\eqref{eqn:lambda4}, while $\mathcal{R}_{0b}^{T1}>1$ gives $\det(\mathcal{J}_{2})<0$ by~\eqref{eqn:det_J2x2}, implying one positive and one negative eigenvalue in the $2\times2$ block. In either case $x_{ab}^{0}$ is unstable.
\end{proof}

\section{Analysis of Quasi-Bacteria-Free Equilibrium} \label{app:QBFE}

\subsection{Proof of Lemma \ref{lem:QBFE_existence}}
\label{app:proof_QBFE_existence}
\begin{proof}
Setting the right-hand sides of~\eqref{eqn:ReducedModel_NoBacterial} to zero gives steady-state system:

\begin{equation}
  \begin{aligned}
    \mu - \beta\, x_{is}\, x_{ss} + \iota\, x_{rs} - \mu\, x_{ss} & = 0,\\
    \beta\, x_{is}\, x_{ss} - j_4\, x_{is}& = 0,\\
    \gamma\, x_{is} - j_8\, x_{rs} &= 0.
    \end{aligned}
    \label{eqn:QBFE_ss1-ss3}
\end{equation}
\textbf{Case 1} ($x_{is}=0$): equations give $x_{rs}=0$, $x_{ss}=1$,
recovering the DFE only.\\
\textbf{Case 2} ($x_{is}>0$): dividing the second equation of ~\eqref{eqn:QBFE_ss1-ss3} by $x_{is}$ gives $x_{ss}^{bf}=j_4/\beta=1/\mathcal{R}_{0a}^{V1}$, which lies in $(0,1)$ if and only if $\mathcal{R}_{0a}^{V1}>1$. The third equation in ~\eqref{eqn:QBFE_ss1-ss3} gives $x_{rs}^{bf}=(\gamma/j_8)x_{is}^{bf}$. Substituting into the normalization constraint $x_{ss}+x_{is}+x_{rs}=1$ and solving yields $x_{is}^{bf}=[j_8/(j_4+\iota)](1-1/\mathcal{R}_{0a}^{V1})>0$ if and only if $\mathcal{R}_{0a}^{V1}>1$.
\end{proof}

\subsection{Proof of Lemma \ref{lem:QBFE_stability}}
\label{app:proof_QBFE_stability}
\begin{proof}
The characteristic polynomial of~\eqref{eqn:J_bf_1} is $P_1^{bf}(\lambda)=\lambda^3+a_1\lambda^2+a_2\lambda+a_3$ with
\begin{align*}
  a_1 &= \frac{\mathcal{R}_{0a}^{V1}j_8}{j_4+\iota}
          \!\left(1-\frac{1}{\mathcal{R}_{0a}^{V1}}\right)+j_8+\mu, \\
  a_2 &= \frac{\beta j_8}{j_4+\iota}
          \!\left(1-\frac{1}{\mathcal{R}_{0a}^{V1}}\right)
          \!\left(\frac{1}{j_8}+\frac{1}{j_4}\right)+\frac{\mu}{j_4}, \\
  a_3 &= j_8\beta\mu\!\left(1-\frac{1}{\mathcal{R}_{0a}^{V1}}\right).
\end{align*}
All three are positive when $\mathcal{R}_{0a}^{V1}>1$.  Expanding $a_1a_2$ confirms $a_1a_2>a_3$ under the same condition. The Routh--Hurwitz criterion then gives local asymptotic stability.
\end{proof}

\subsection{Proof of Theorem \ref{thm:ReducedModel_NoBacterial_4ward}}
\label{app:proof_bifurcation}
\begin{proof}
Set $\beta^*=j_4$ so that $\mathcal{R}_{0a}^{V1}=1$. Rename $(x_{ss},x_{is},x_{rs})\mapsto(x_1,x_2,x_3)$ and define $y(x,\beta)=dx/dt\in\mathcal{C}^2(\mathbb{R}^3\times\mathbb{R})$. The Jacobian at the DFE when $\beta=\beta^*$ is
\[
  D^{bf}_{\beta^*}=
  \begin{pmatrix}-\mu&\iota&-\beta^*\\0&-j_8&\gamma\\0&0&0\end{pmatrix},
\]
which has a simple zero eigenvalue.  The right and left eigenvectors corresponding to zero, normalized so that $v\cdot w=1$, are
\[
  w=(-(\gamma+\iota+\mu),\,j_8,\,\gamma)^{\top},\qquad
  v=(0,1,0).
\]
Computing the second-order directional derivatives (Theorem~4.1 of \cite{castillo2004dynamical}) gives $a=-j_8\beta^*/(j_4+\iota)<0$ and $b=1>0$. Since $a<0$ and $b>0$, the bifurcation is supercritical (forward).
\end{proof}

	\subsection{Co-infected matrices}\label{ssst:Coinfected matrices}
	
	\begin{align*}
		F_1 =\left(\begin{array}{rrrrrr}
			{x_{ss}^{bf}} {\beta_{c}} & {x_{ss}^{bf}} {\beta_{c}} & {x_{ss}^{bf}} {\beta_{c}} & {x_{ss}^{bf}} {\beta_{I}} & {x_{ss}^{bf}} {\beta_{I}} & {x_{ss}^{bf}} {\beta_{I}} \\
			{x_{is}^{bf}} {\beta_{c}} & {x_{is}^{bf}} {\beta_{c}} & {x_{is}^{bf}} {\beta_{c}} & {x_{is}^{bf}} {\beta_{I}} & {x_{is}^{bf}} {\beta_{I}} & {x_{is}^{bf}} {\beta_{I}} \\
			{x_{rs}^{bf}} {\beta_{c}} & {x_{rs}^{bf}} {\beta_{c}} & {x_{rs}^{bf}} {\beta_{c}} & {x_{rs}^{bf}} {\beta_{I}} & {x_{rs}^{bf}} {\beta_{I}} & {x_{rs}^{bf}} {\beta_{I}} \\
			0 & 0 & 0 & 0 & 0 & 0 \\
			0 & 0 & 0 & 0 & 0 & 0 \\
			0 & 0 & 0 & 0 & 0 & 0
		\end{array}\right), 	F2=\left(\begin{array}{rrrr}
			{x_{ss}^{vf}} {\beta} & {x_{ss}^{vf}} {\beta} & {x_{ss}^{vf}} {\beta} & {x_{ss}^{vf}} {\beta} \\
			{x_{sc}^{vf}} {\beta} & {x_{sc}^{vf}} {\beta} & {x_{sc}^{vf}} {\beta} & {x_{sc}^{vf}} {\beta} \\
			{x_{si}^{vf}} {\beta} & {x_{si}^{vf}} {\beta} & {x_{si}^{vf}} {\beta} & {x_{si}^{vf}} {\beta} \\
			{x_{sr}^{vf}} {\beta} & {x_{sr}^{vf}} {\beta} & {x_{sr}^{vf}} {\beta} & {x_{sr}^{vf}} {\beta}
		\end{array}\right),
	\end{align*}
	and
	\begin{align*}
		V_1= \left(\begin{array}{rrrrrr}
			j_{1} & 0 & -{\iota} & 0 & 0 & 0 \\
			0 & j_{5} & 0 & 0 & 0 & 0 \\
			0 & -{\gamma} & j_{9} & 0 & 0 & 0 \\
			-{\kappa_{S}} & 0 & 0 & j_{2} & 0 & -{\iota} \\
			0 & -{\kappa_{I}} & 0 & 0 & j_{6} & 0 \\
			0 & 0 & -{\kappa_{R}} & 0 & -{\gamma} & j_{10}
		\end{array}\right), 
		V2= \left(\begin{array}{rrrr}
			j_{4} & 0 & 0 & -{\iota_{b}} \\
			0 & j_{5} & 0 & 0 \\
			0 & -{\kappa_{I}} & j_{6} & 0 \\
			0 & -{\theta} & -{\gamma_{b}} & j_{7}
		\end{array}\right)
	\end{align*}
    
\subsection{Proof of Lemma \ref{lem:J3_stability}}
\label{app:proof_J3_stability}
\begin{proof}
With $x_{is}^{bf}$ from~\eqref{eqn:Bacteria-free fixed} and $\mathcal{R}_{0a}^{V1}>1$, all of $c_1$, $c_2$, $c_3$ are strictly positive.  Direct expansion gives
\[
  c_1c_2-c_3
  = (x_{is}^{bf}\beta+j_{11}+j_7)(c_2+j_3^2) > 0,
\]
so the Routh--Hurwitz criterion is satisfied.
\end{proof}

\subsection{Phase-Space Dynamics of the Bacteria-Free Subsystem}\label{app:QBFE_phase_space}
\begin{figure}[pos=ht] 
  \centering
\includegraphics[width=.85\linewidth]{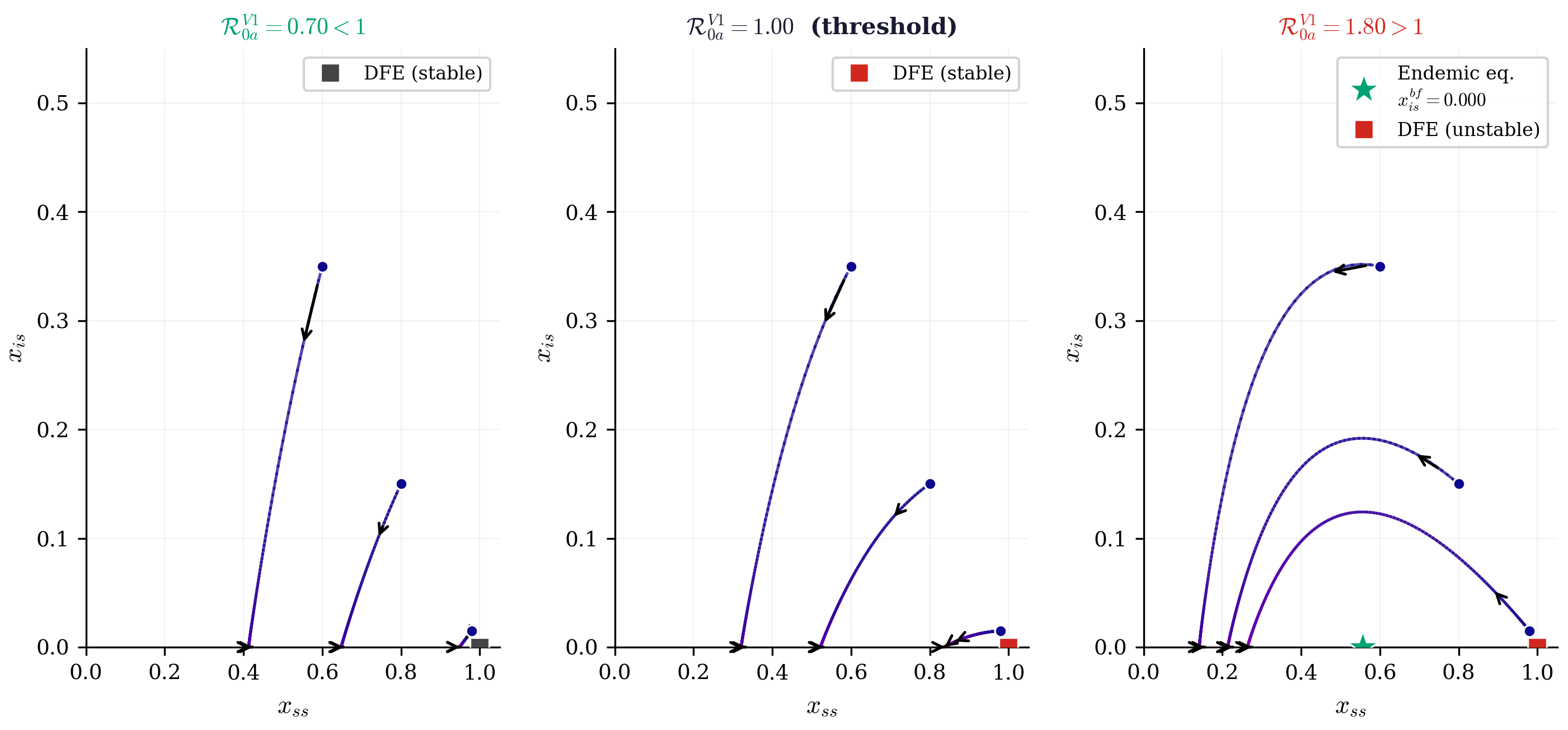} 
 \caption{Phase-plane trajectories (\(x_{ss}\) vs. \(x_{is}\)) across subcritical (\(\mathcal{R}_{0a}^{V1} = 0.70\)), critical (\(1.00\)), and epidemic (\(1.80\)) regimes evaluated at baseline values (\(\gamma = 0.33\), \(\mu = 0.0001\), \(\iota = 0.00\)). Red squares and green stars denote the disease-free and stable endemic equilibria respectively, tracking global path convergence toward the active attractor point (\(x_{ss}^{bf} \approx 0.556\)).}
\label{fig:QBFE_Phase_Plane}
\end{figure}

Figure \ref{fig:QBFE_Phase_Plane} visualizes the 2D phase-space trajectories (\(x_{ss}\) vs. \(x_{is}\)) of the bacteria-free subsystem across three distinct epidemic transmission regimes. In the subcritical and critical states (\(\mathcal{R}_{0a}^{V1} = 0.70\) and 1.00), all trajectory paths slope downward to the coordinate floor and terminate at the stable disease-free equilibrium at (1.0, 0.0), highlighting that initial infections will naturally die out. Once transmission crosses the critical threshold (\(\mathcal{R}_{0a}^{V1} = 1.80\)), the disease-free point at (1.0, 0.0) loses its stability (red square), and a stable endemic equilibrium point emerges (green star). In this epidemic state, trajectory paths starting from separate initial conditions loop upward to track historical infection waves before settling at the endemic target point on the horizontal axis (\(x_{ss} \approx 0.556\)), demonstrating that the virus successfully maintains a persistent, low-level circulation within the host population.

\subsection{Time-Series Simulations and Compartmental Convergence} \label{app:QBFE_time_series}

\begin{figure}[pos=ht] 
  \centering
\includegraphics[width=.85\linewidth]{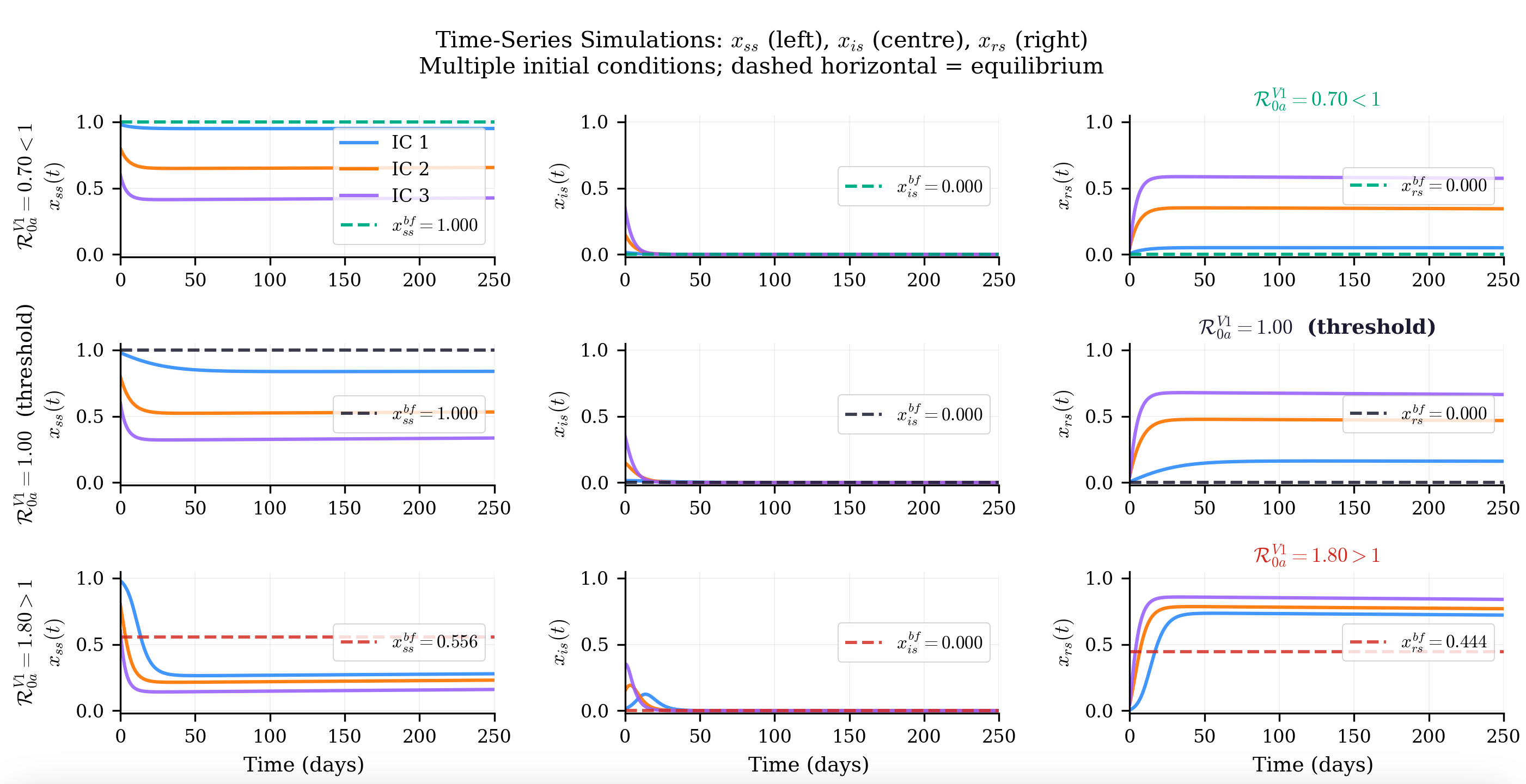}
\caption{250-day time series of $x_{ss}$, $x_{is}$, and $x_{rs}$ from multiple initial conditions, at $\mathcal{R}_{0a}^{V1}=1.80$. Dashed lines mark the analytical equilibrium (Eq.~\eqref{eqn:Bacteria-free fixed}); trajectories converge while conserving $x_{ss}+x_{is}+x_{rs}=1$.}
\label{fig:QBFE_Time_series}
\end{figure}

Figure \ref{fig:QBFE_Time_series} illustrates the explicit time-series trajectories of the host population compartments—susceptible (\(x_{ss}(t)\), left), infected (\(x_{is}(t)\), center), and recovered (\(x_{rs}(t)\), right)—across three transmission regimes over a 250-day period. In the subcritical and critical regimes (\(\mathcal{R}_{0a}^{V1} = 0.70\) and \(1.00\)), the infected fraction (\(x_{is}\)) in the center panel rapidly decays to zero from all initial conditions, while the susceptible and recovered groups stabilize as parallel straight lines entirely dependent on their starting ratios. This behavior confirms that below the epidemic threshold, no secondary transmission occurs. Conversely, in the epidemic regime (\(\mathcal{R}_{0a}^{V1} = 1.80\)), the system exhibits a dynamic epidemiological wave; the infected curves peak early before dropping, forcing all trajectories to break away from their starting conditions and converge onto the dashed equilibrium lines. Ultimately, the susceptible population drops to its steady-state target (\(x_{ss}^{bf} = 0.556\)) and the recovered pool climbs to its endemic baseline (\(x_{rs}^{bf} = 0.444\)), visually confirming the global stability of the endemic equilibrium state.

\section{Analysis of Quasi-Virus-Free-Equilibrium}
\subsection{Proof of Lemma \ref{lem:Unique_QVFE}}
\label{app:proof_QVFE_existence}
\begin{proof}
Set right-hand sides of~\eqref{eqn:ReducedModel_NoVirus} to zero. Case $x_{sc}=0$: gives $x_{si}=x_{sr}=0$, $x_{ss}=1$ (DFE only). Case $x_{sc}>0$: from the $x_{sc}$ equation, $x_{ss}^{vf}=j_1/(\beta_c+\beta_I\kappa_S/j_2)=1/\mathcal{R}_{0b}^{T1}$, which lies in $(0,1)$ iff $\mathcal{R}_{0b}^{T1}>1$. The remaining components follow from the $x_{si}$ and $x_{sr}$ equations together with the normalization constraint.
\end{proof}

\subsection{Proof of Theorem \ref{thm:ReducedModel_NoVirus_Backward}}
\label{app:proof_QVFE_bifurcation}
\begin{proof}
Apply Theorem~4.1 of \cite{castillo2004dynamical} with bifurcation parameter $\beta_c^*=j_1$ (value at $\mathcal{R}_{0b}^{T1}=1$). The Jacobian at the DFE when $\beta_c=\beta_c^*$ has a simple zero eigenvalue.  Computing right and left eigenvectors and the second-order directional derivatives yields formula~\eqref{eqn:QVFE_ab}. Since the numerator of $a$ is a sum of products of positive parameters, it is strictly positive.  The sign of $a$ is therefore determined entirely by the sign of $\beta_c-(\kappa_S+\mu+\theta)$.
\end{proof}

\subsection{Stability of $\mathcal{J}^{vf}_{3}$}
\label{app:proof_QVFE_J3_stability}
\begin{proof}
$\mathcal{J}^{vf}_{3}$ is lower triangular with diagonal entries
$-(x_{si}^{vf}\beta_I+x_{sc}^{vf}\beta_c+j_8)$, $-j_9$, $-j_{10}$, $-j_{11}$, all strictly negative when $\mathcal{R}_{0b}^{T1}>1$ (so $x_{sc}^{vf},x_{si}^{vf}>0$) and all parameters positive. Hence all eigenvalues are strictly negative.
\end{proof}

\section{Components of the Coexistence Endemic Equilibrium}
\label{app:EE}

Under the assumption $\iota=\iota_b=0$, setting the right-hand sides of System~\eqref{eqn:MainModel} to zero and solving in terms of the endemic forces of infection~\eqref{eqn:force_of_infections_Endemic} gives
\begin{align}
	x_{ss}^* &= \frac{\Lambda}{\xi_V^*+\xi_B^*+\mu},
	\label{eqn:EE_SS}\\[4pt]
	x_{sc}^* &= \frac{\Lambda\,\xi_B^*}
	{(\xi_V^*+\xi_B^*+\mu)(\xi_V^*+j_1)},
	\label{eqn:EE_SC}\\[4pt]
	x_{si}^* &= \frac{\Lambda\,\xi_B^*\,\kappa_S}
	{(\xi_V^*+\xi_B^*+\mu)(\xi_V^*+j_1)(\xi_V^*+j_2)},
	\label{eqn:EE_SI}\\[4pt]
	x_{sr}^* &= \frac{\Lambda\,\xi_B^*}
	{(\xi_V^*+\xi_B^*+\mu)(\xi_V^*+j_1)(\xi_V^*+j_3)}
	\!\left(\frac{\gamma_b\kappa_S}{\xi_V^*+j_2}+\theta\right),
	\label{eqn:EE_SR}\\[4pt]
	X_{IS}^* &= \frac{\Lambda\,\xi_V^*}
	{(\xi_V^*+\xi_B^*+\mu)(\xi_B^*+j_4)},
	\label{eqn:EE_IS}\\[4pt]
	x_{ic}^* &= \frac{\Lambda\,\xi_V^*\,\xi_B^*}
	{(\xi_V^*+\xi_B^*+\mu)\,j_5}
	\!\left(\frac{1}{\xi_V^*+j_1}+\frac{1}{\xi_B^*+j_4}\right),
	\label{eqn:EE_IC}\\[4pt]
	X_{II}^* &= \frac{\Lambda\,\xi_V^*\,\xi_B^*}
	{(\xi_V^*+\xi_B^*+\mu)\,j_6}
	\!\left(\frac{\kappa_I}{j_5}
	\!\left(\frac{1}{\xi_V^*+j_1}+\frac{1}{\xi_B^*+j_4}\right)
	+\frac{\kappa_S}{(\xi_V^*+j_1)(\xi_V^*+j_2)}\right),
	\label{eqn:EE_II}\\[4pt]
	x_{ir}^* &= \frac{x_{sr}^*\,\xi_V^*+\theta\,x_{ic}^*+\gamma_b\,X_{II}^*}
	{j_7},
	\label{eqn:EE_IR}\\[4pt]
	x_{rs}^* &= \frac{\Lambda\,\xi_V^*\,\gamma}
	{(\xi_V^*+\xi_B^*+\mu)(\xi_B^*+j_4)(\xi_B^*+j_8)},
	\label{eqn:EE_RS}\\[4pt]
	x_{rc}^* &= \frac{\Lambda\,\xi_V^*\,\xi_B^*\,\gamma}
	{(\xi_V^*+\xi_B^*+\mu)\,j_9}
	\!\left(\frac{1}{j_5}
	\!\left(\frac{1}{\xi_V^*+j_1}+\frac{1}{\xi_B^*+j_4}\right)
	+\frac{1}{(\xi_B^*+j_4)(\xi_B^*+j_8)}\right),
	\label{eqn:EE_RC}\\[4pt]
	x_{ri}^* &= \frac{\Lambda\,\xi_V^*\,\xi_B^*\,\gamma}
	{(\xi_V^*+\xi_B^*+\mu)\,j_{10}}\,\mathbb{P},
	\label{eqn:EE_RI}
\end{align}
where
\begin{equation}
	\mathbb{P} =
	\frac{\kappa_R}{j_9}
	\!\left(\frac{1}{j_5}
	\!\left(\frac{1}{\xi_V^*+j_1}+\frac{1}{\xi_B^*+j_4}\right)
	+\frac{1}{(\xi_B^*+j_4)(\xi_B^*+j_8)}\right)
	+\frac{1}{j_6}
	\!\left(\frac{\kappa_I}{j_5}
	\!\left(\frac{1}{\xi_V^*+j_1}+\frac{1}{\xi_B^*+j_4}\right)
	+\frac{\kappa_S}{(\xi_V^*+j_1)(\xi_V^*+j_2)}\right),
	\label{eqn:EE_P}
\end{equation}
and finally
\begin{equation}
	x_{rr}^* = \frac{\gamma\,x_{ir}^*+\theta\,x_{rc}^*+\gamma_b\,x_{ri}^*}
	{j_{11}}.
	\label{eqn:EE_RR}
\end{equation}
The endemic forces of infection are
\begin{equation}
	\xi_V^* = \frac{\beta}{N^*}(X_{IS}^*+x_{ic}^*+X_{II}^*+x_{ir}^*),
	\qquad
	\xi_B^* = \frac{\beta_c}{N^*}(x_{sc}^*+x_{ic}^*+x_{rc}^*)
	+\frac{\beta_I}{N^*}(x_{si}^*+X_{II}^*+x_{ri}^*),
	\label{eqn:force_of_infections_Endemic}
\end{equation}
where $N^*=\sum_{a,b}X_{ab}^*$. These constitute a nonlinear system in $(\xi_V^*,\xi_B^*)$ whose positive solutions determine the CEE.

\section{Eigenvectors of $\mathcal{J}_{\mathcal{T}}$ at the DFE}
\label{app:proof_CEE_eigenvectors}
\begin{proof}
The right eigenvector $\mathbf{w}$ is found by solving $\mathcal{J}_{\mathcal{T}}\mathbf{w}=\mathbf{0}$. Rows 6, 7, 8 of $\mathcal{J}_{\mathcal{T}}$ give $w_6=w_7=w_8=0$ immediately (since the corresponding diagonal blocks are $-j_5,-j_6,-j_7$ with no other non-zero entries in those columns at $\beta=\beta^*$). Similarly rows 10, 11, 12 give $w_{10}=w_{11}=w_{12}=0$. Row 5 (equation for $x_5$) with $w_6=w_7=w_8=0$ gives $\beta^* w_6 + \beta^* w_7 + (\beta^*+\iota_b)w_8 = j_4 w_5$, ... [continue to derive all components in order] The left eigenvector $\mathbf{v}$ is found analogously from $\mathbf{v}\mathcal{J}_{\mathcal{T}}=\mathbf{0}$.
\end{proof}

\section{Second Derivatives for the Bifurcation Coefficient $a$}
\label{app:Bif_Coeff}

The non-zero second partial derivatives of System~\eqref{eqn:MainModel_NewNotation} evaluated at the DFE $x_{ab}^{0}$ are listed below, grouped by the equation $y_k$ to which they belong. All unlisted second derivatives are zero at $x_{ab}^{0}$.

\paragraph{From $y_2$ (bacterial carrier equation).} Differentiating $y_2 = -x_2\xi_V + x_1\xi_B + \iota x_{10} - j_1 x_2$ with respect to pairs involving $x_1$ (source of bacterial infection) and $x_5,x_6,x_7,x_8$ (viral infectives via $\xi_V$):
\begin{alignat*}{2}
	\frac{\partial^2 y_2}{\partial x_1\partial x_2}(x_{ab}^{0})
	&= \frac{\partial^2 y_2}{\partial x_1\partial x_6}(x_{ab}^{0})
	= \frac{\partial^2 y_2}{\partial x_1\partial x_{10}}(x_{ab}^{0})
	&&= \beta_c,\\
	\frac{\partial^2 y_2}{\partial x_1\partial x_3}(x_{ab}^{0})
	&= \frac{\partial^2 y_2}{\partial x_1\partial x_7}(x_{ab}^{0})
	= \frac{\partial^2 y_2}{\partial x_1\partial x_{11}}(x_{ab}^{0})
	&&= \beta_I,\\
	\frac{\partial^2 y_2}{\partial x_2\partial x_5}(x_{ab}^{0})
	&= \frac{\partial^2 y_2}{\partial x_2\partial x_6}(x_{ab}^{0})
	= \frac{\partial^2 y_2}{\partial x_2\partial x_7}(x_{ab}^{0})
	= \frac{\partial^2 y_2}{\partial x_2\partial x_8}(x_{ab}^{0})
	&&= -\beta^*.
\end{alignat*} 

\paragraph{From $y_3$ (bacterial infectious equation).}
Differentiating $y_3 = -x_3\xi_V + \kappa_S x_2 + \iota x_{11} - j_2 x_3$:
\begin{equation*}
	\frac{\partial^2 y_3}{\partial x_3\partial x_5}(x_{ab}^{0})
	= \frac{\partial^2 y_3}{\partial x_3\partial x_6}(x_{ab}^{0})
	= \frac{\partial^2 y_3}{\partial x_3\partial x_7}(x_{ab}^{0})
	= \frac{\partial^2 y_3}{\partial x_3\partial x_8}(x_{ab}^{0})
	= -\beta^*.
\end{equation*}

\paragraph{From $y_5$ (viral infectious equation).}
Differentiating $y_5 = x_1\xi_V - x_5\xi_B + \iota_b x_8 - j_4 x_5$:
\begin{alignat*}{2}
	\frac{\partial^2 y_5}{\partial x_1\partial x_5}(x_{ab}^{0})
	&= \frac{\partial^2 y_5}{\partial x_1\partial x_6}(x_{ab}^{0})
	= \frac{\partial^2 y_5}{\partial x_1\partial x_7}(x_{ab}^{0})
	= \frac{\partial^2 y_5}{\partial x_1\partial x_8}(x_{ab}^{0})
	&&= \beta^*,\\
	\frac{\partial^2 y_5}{\partial x_2\partial x_5}(x_{ab}^{0})
	&= \frac{\partial^2 y_5}{\partial x_5\partial x_6}(x_{ab}^{0})
	= \frac{\partial^2 y_5}{\partial x_5\partial x_{10}}(x_{ab}^{0})
	&&= -\beta_c,\\
	\frac{\partial^2 y_5}{\partial x_3\partial x_5}(x_{ab}^{0})
	&= \frac{\partial^2 y_5}{\partial x_5\partial x_7}(x_{ab}^{0})
	= \frac{\partial^2 y_5}{\partial x_5\partial x_{11}}(x_{ab}^{0})
	&&= -\beta_I.
\end{alignat*}

\paragraph{From $y_6$ (co-infected carrier equation).}
Differentiating $y_6 = x_2\xi_V + x_5\xi_B - j_5 x_6$:
\begin{alignat*}{2}
	\frac{\partial^2 y_6}{\partial x_2\partial x_5}(x_{ab}^{0})
	&&&= \beta^*+\beta_c,\\
	\frac{\partial^2 y_6}{\partial x_2\partial x_6}(x_{ab}^{0})
	&= \frac{\partial^2 y_6}{\partial x_2\partial x_7}(x_{ab}^{0})
	= \frac{\partial^2 y_6}{\partial x_2\partial x_8}(x_{ab}^{0})
	&&= \beta^*,\\
	\frac{\partial^2 y_6}{\partial x_5\partial x_6}(x_{ab}^{0})
	&= \frac{\partial^2 y_6}{\partial x_5\partial x_{10}}(x_{ab}^{0})
	&&= \beta_c,\\
	\frac{\partial^2 y_6}{\partial x_3\partial x_5}(x_{ab}^{0})
	&= \frac{\partial^2 y_6}{\partial x_5\partial x_7}(x_{ab}^{0})
	= \frac{\partial^2 y_6}{\partial x_5\partial x_{11}}(x_{ab}^{0})
	&&= \beta_I.
\end{alignat*}

\paragraph{From $y_7$ (co-infected active equation).}
Differentiating $y_7 = x_3\xi_V + \kappa_I x_6 - j_6 x_7$:
\begin{equation*}
	\frac{\partial^2 y_7}{\partial x_3\partial x_5}(x_{ab}^{0})
	= \frac{\partial^2 y_7}{\partial x_3\partial x_6}(x_{ab}^{0})
	= \frac{\partial^2 y_7}{\partial x_3\partial x_7}(x_{ab}^{0})
	= \frac{\partial^2 y_7}{\partial x_3\partial x_8}(x_{ab}^{0})
	= \beta^*.
\end{equation*}

\paragraph{From $y_8$ (co-infected recovered equation).}
Differentiating $y_8 = x_4\xi_V + \gamma_b x_7 + \theta x_6 - j_7 x_8$:
\begin{equation*}
	\frac{\partial^2 y_8}{\partial x_4\partial x_5}(x_{ab}^{0})
	= \frac{\partial^2 y_8}{\partial x_4\partial x_6}(x_{ab}^{0})
	= \frac{\partial^2 y_8}{\partial x_4\partial x_7}(x_{ab}^{0})
	= \frac{\partial^2 y_8}{\partial x_4\partial x_8}(x_{ab}^{0})
	= \beta^*.
\end{equation*}

\paragraph{Derivative with respect to the bifurcation parameter.}
\begin{equation*}
	\frac{\partial^2 y_5}{\partial x_5\,\partial\beta^*}(x_{ab}^{0})
	= \frac{\partial^2 y_5}{\partial x_6\,\partial\beta^*}(x_{ab}^{0})
	= \frac{\partial^2 y_5}{\partial x_7\,\partial\beta^*}(x_{ab}^{0})
	= \frac{\partial^2 y_5}{\partial x_8\,\partial\beta^*}(x_{ab}^{0})
	= 1.
\end{equation*}

\section{Sensitivity Analysis Pipeline}\label{app:LHS_PRCC}

We used Latin Hypercube Sampling (LHS) paired with Partial Rank Correlation Coefficients (PRCC), following the standard approach of \cite{blower1994sensitivity,marino2008methodology}. The $K=12$ parameters
\[
  \mathbf{P} = \{\beta, \gamma, \beta_c, \gamma_b, \theta, \kappa_S, \iota, \iota_b, \mu, \beta_I, \kappa_I, \kappa_R\}
\]
were each partitioned into $N=5{,}000$ equiprobable intervals, sampled once per interval, and randomly permuted across rows to form the LHS matrix $\mathbf{X}\in\mathbb{R}^{N\times K}$. The model was evaluated at each row to produce the seven output vectors listed in Section~\ref{sec:Sensitivity}. Because the QBFE and QVFE fractions are undefined below their respective reproduction thresholds, rows producing non-finite outputs were dropped, leaving $N_m\leq N$ valid samples.

Since the input--output relationships are nonlinear but monotone, we rank-transformed $\mathbf{X}_m$ and each output $\mathbf{Y}_m$ before computing partial correlations. For each parameter $X_j$, the PRCC $r_{X_j,Y\cdot\mathbf{Z}}$ was obtained from the precision matrix $B=C^{-1}$ of the combined rank-correlation matrix $C\in\mathbb{R}^{(K+1)\times(K+1)}$ built from $[\mathbf{X}_r,\mathbf{Y}_r]$:
\begin{equation}
  r_{X_j, Y\cdot\mathbf{Z}} = \frac{-B_{j,K+1}}{\sqrt{B_{j,j}\,B_{K+1,K+1}}}.
\end{equation}
Significance was assessed via $t = r_{X_j,Y\cdot\mathbf{Z}}\sqrt{\mathrm{df}/(1-r_{X_j,Y\cdot\mathbf{Z}}^2)}$ with $\mathrm{df}=N_m-K-1$, using a two-tailed Student's $t$-test; correlations with $p\geq0.05$ were treated as not statistically significant and masked in Figure~\ref{fig:prcc_heatmap}.

\bibliographystyle{model1-num-names}

\bibliography{references}

\end{document}